\documentclass[draft]{amsart}
\usepackage{amssymb,amsmath,amsthm}
\usepackage{enumerate}
\usepackage[draft=false]{hyperref}
\usepackage{xcolor}
\hypersetup{
	colorlinks,
	linkcolor={red!50!black},
	citecolor={blue!50!black},
	urlcolor={blue!80!black}
}

\title{Invariant measures in simple and in small theories}
\date{\today}

\author[Chernikov, Hrushovski, Kruckman, Krupi\'nski, Moconja, Pillay, Ramsey]{Artem Chernikov}
\address[A.\ Chernikov]{UCLA}
\email{chernikov@math.ucla.edu}

\author[]{Ehud Hrushovski}
\address[E.\ Hrushovski]{University of Oxford}
\email{ehud.hrushovski@maths.ox.ac.uk}

\author[]{Alex Kruckman}
\address[A.\ Kruckman]{Wesleyan University}
\email{akruckman@wesleyan.edu}
	
\author[]{Krzysztof Krupi\'nski}
\address[K.\ Krupi\'nski]{\parbox{\linewidth}{Instytut Matematyczny, Uniwersytet Wroc\l awski\\ Pl. Grunwaldzki 2, 50-384, Wroc\l aw\\
\url{https://orcid.org/0000-0002-2243-4411}}}
\email{kkrup@math.uni.wroc.pl}	

\author[]{Slavko Moconja}
\address[S.\ Moconja]{University of Belgrade}
\email{slavko@matf.bg.ac.rs}

\author[]{Anand Pillay}
\address[A.\ Pillay]{University of Notre Dame}
\email{apillay@nd.edu}

\author[]{Nicholas Ramsey}
\address[N.\ Ramsey]{UCLA}
\email{nickramsey@math.ucla.edu}

\thanks{The first author is supported by the NSF CAREER grant DMS-1651321 and by a Simons Fellowship. 
The fourth author is supported by the Narodowe Centrum Nauki grants nos. 2016/22/E/ST1/00450 and 2018/31/B/ST/00357. 
The fifth author is supported by the Science Fund of the Republic of Serbia,
grant 7750027--SMART.
The sixth author is supported by NSF grants DMS-1665035,  DMS-1760212, and DMS-2054271.}

\newtheorem{Theorem}{Theorem}[section]
\newtheorem{Proposition}[Theorem]{Proposition}
\newtheorem{Lemma}[Theorem]{Lemma}
\newtheorem{Corollary}[Theorem]{Corollary}
\newtheorem{Fact}[Theorem]{Fact}
\newtheorem*{Claim*}{Claim}

\newtheorem{unutraLemma}{Lemma}
\newenvironment{uLemma}[1]
  {\renewcommand\theunutraLemma{#1}\unutraLemma}
  {\endunutraLemma}

\theoremstyle{definition}
\newtheorem{Definition}[Theorem]{Definition}
\newtheorem{Axiom}{Axiom Schema}

\theoremstyle{remark}
\newtheorem{Remark}[Theorem]{Remark}
\newtheorem{Example}[Theorem]{Example}

\newenvironment{explanation}{

\begin{proof}

}{
\end{proof}
}

\newenvironment{proof*}[1]{
\noindent{\em #1.}
}{
\qed
}

\newcommand{\R}{\mathbb R}   

\newcommand{\Q}{\mathbb Q}  

\newcommand{\Z}{\mathbb Z}  

\newcommand{\N}{\mathbb N}

\begin{document}
\maketitle

\begin{abstract} We give examples of (i) a simple theory with a formula (with parameters) which does not fork over $\emptyset$ but has $\mu$-measure $0$ for every automorphism invariant Keisler measure $\mu$, and  (ii) a definable group $G$ in a simple theory such that $G$ is not definably amenable, i.e. there is no translation invariant Keisler measure on $G$. 

We also discuss paradoxical decompositions both in the setting of discrete groups and of definable groups, and prove some positive results about small theories, including the definable amenability of definable groups. 
\end{abstract}

\section{Introduction and preliminaries}
We begin with an introduction for a general audience.  The paper is about {\em amenability} in model-theoretic environments, with  both nonexistence and existence theorems. The expression ``amenability"  often refers to the existence of a finitely additive probability measure $\mu$ on some suitable collection ${\mathcal B}$ of subsets of a given set $X$, which is invariant under a certain action of a certain group $G$.  When $X=G$, $\mathcal B$ is the collection of {\em all} subsets of $G$,  and the action is the action of $G$ on $\mathcal B$ by left translation, then we obtain precisely the ``classical" notion of amenability of $G$ as a discrete group.  Remaining in this context, one could replace the Boolean algebra of {\em all} subsets of $G$ by some other Boolean algebra of subsets of $G$ invariant under left translation, and ask for amenability with respect to the new Boolean algebra. In some interesting examples one obtains strikingly different behaviour when passing to  natural and reasonably rich Boolean algebras. For example the free group $F_{2}$ on two generators is not amenable as a discrete group, but if we choose instead the Boolean algebra $\mathcal B$ to be the collection of subsets of $F_{2}$ which are {\em definable} (with parameters) in the structure $(F_{2},\times)$, then not only do we get amenability, but ``unique ergodicity": there is a unique invariant measure which is moreover $\{0,1\}$-valued.  This is a consequence of the fact that  the first order theory  $Th((F_{2},\times))$ of the structure $(F_{2},\times)$ has a property called {\em stability}, which can be summed up by the statement that ``any stable group is (uniquely) definably amenable".  In addition to the free group, all commutative groups and all algebraic groups over algebraically closed fields are stable.   A more general class of first order theories, the class of  so-called {\em  simple theories} was defined and studied beginning in the 1980's, often in the context of  specific examples of independent interest such as {\em pseudofinite fields} (logical limits of finite fields). Early applications were to algebraic groups over finite fields \cite{HP-algebraicgroups}.  Groups definable in pseudofinite fields are definably amenable witnessed by a ``nonstandard counting measure". It was asked around ten years ago whether groups definable in any  simple theory are definably amenable.  One of our main theorems appearing in Section \ref{sec 3}  (as in (ii) of the abstract) is a counterexample. 

We now give some background for (i) in the abstract, which on the face of it, may seem less accessible to the general reader. 
The context, implicit in the paragraph above,  is a structure $M$ in the sense of model theory, namely an underlying set which we also call $M$, equipped with a collection $D$ of distinguished subsets of various Cartesian powers $M^{n}$ of $M$, including the diagonal $\subset M^{2}$.   The automorphism group $Aut(M)$ is the group of permutations of $M$ which fix setwise each of the distinguished sets.  Closing under the operations of finite Boolean combination, and projection (from $M^{n+1}$ to $M^{n}$),  we obtain the class $D_{1}$ of $\emptyset$-definable sets. For $X\subseteq M^{n+k}$  in $D_{1}$, and ${\bar a}\in M^{k}$, let $X_{\bar a} = \{{\bar b}\in M^{n}: ({\bar b}, {\bar a})\in X\}$.  These various $X_{\bar a}$ (as $X$ and $\bar a$ vary), are called the definable (with parameters) sets in the structure $M$. $Aut(M)$ acts on the collection of definable sets.  We fix some ambient Cartesian power $M^{n}$ of $M$, and consider the Boolean algebra  ${\mathcal B}$ of definable subsets of $M^{n}$, again acted on by $Aut(M)$. We make an additional assumption on $M$ (saturation) ensuring that $Aut(M)$ is ``large" in a suitable sense.  One of the recent waves of connections between model theory and combinatorics, specifically \cite{Hrushovski-approximate}, was largely based on  an analogy between two kinds of  ($Aut(M)$-)invariant ideals of ${\mathcal B}$: the ``forking ideal" $I_{f}$ in the case that the first order theory $Th(M)$ is simple (see below for details and definitions)  and for any invariant finitely additive probability measure $\mu$ on ${\mathcal B}$, such as the nonstandard counting measure when $M$ is pseudofinite, the $\mu$-measure $0$ ideal $I_{\mu}$. We always have that   $I_{f} \subseteq I_{\mu}$, and it was an open question whether for simple theories $I_{f}$ is precisely the intersection of the $I_{\mu}$ as $\mu$ varies over all  invariant measures.   We answer this question negatively in this paper. The main example is constructed in Section \ref{sec 2}, producing a theory with many invariant measures, and a formula which is in $I_{\mu}$ for all $\mu$ but not in $I_{f}$.
On the other hand, a corollary of the main theorem in Section \ref{sec 3}, is the existence of a simple theory and a ``sort" on which there are {\em no} invariant measures, giving another route to a negative answer to the question. 

Another aspect of the paper, which is made explicit in Section \ref{sec 4}, concerns the  ``paradoxical decomposition" obstructions to amenability in the various senses. We are interested in  definable versions of paradoxical decompositions, and which model theoretic properties of theories $T$ are incompatible with definable paradoxical decompositions.  Various results are obtained including the definable amenability of definable groups in ``small" theories (where the Boolean algebras of $\emptyset$-definable sets admit a Cantor-Bendixon analysis).\\

We now pass on to a more technical introduction, for readers familiar with model theory.

In stable theories,  Keisler measures are very well understood,  originating in \cite{Keisler}.  There was a comprehensive study of Keisler measures in $NIP$ theories, starting with  \cite{NIPI}, \cite{NIPII}, \cite{NIPIII}. There are indications that many of these results fail outside of NIP (see \cite{CGH}). 
It is  very natural to ask what happens in simple theories.
The main thrust of the current paper is to give counterexamples to some of these questions, as in the abstract. Another aspect of the paper is to give some positive results in the case of countable small theories. 

Partly as motivation we will, in this introduction,  discuss and recall what is known about Keisler measures and forking in general, as well as in stable and $NIP$ theories and then  state the questions which are answered in the body of this paper. 
  
Our model-theoretic notation is standard.  Models will be denoted by $M, N, \ldots$ and subsets (sets of parameters) by $A, B, \ldots$.  If and when we work with a complete theory $T$ then we often work in a sufficiently saturated model, called $\mathfrak C$ or ${\bar M}$; $a,b, \ldots$ refer to tuples in models of $T$ unless we say otherwise or clear from the context. 

The study of {\em stable theories} is connected to {\em categoricity} and is largely due to Shelah \cite{Shelah-book}. There are many other  reference books, including \cite{Pillay-GST}.   In the middle 1990's the machinery of stability theory was extended or generalized to the class of {\em simple theories}  which had been defined earlier by Shelah in \cite{Shelah-simple}. This development was closely connected to and went  in parallel with the concrete analysis of several kinds of structures and theories, including Lie coordinatizable and smoothly approximable  structures (\cite{KLM},  \cite{Cherlin-Hrushovski}), and bounded \emph{PAC} fields (\cite{HP}  and the later published \cite{Hrushovski-PAC}), using tools with a stability theoretic flavour.  In fact Hrushovski's \emph{$S_{1}$-theories} already provided a certain abstract finite rank environment for adapting stability to the more general situations. 
The technical breakthroughs came with Byunghan Kim's thesis \cite{Kim}, \cite{Kim-paper}  followed by \cite{Kim-Pillay}. Kim showed that all the machinery of nonforking independence extended word-for-word from stable theories to simple theories, except for stationarity of types over models (or more generally algebraically closed sets), and \cite{Kim-Pillay}  found the appropriate weak version of the stationarity theory: the Independence Theorem over a model, or more generally for Lascar strong types.  The latter, improved to so-called Kim-Pillay strong types, migrated  and became essential in all of model theory, and also made connections to combinatorics and Lie groups possible, although we still do not know,  whether this level of generality, versus the strong types of Shelah, is really needed in simple theories. 
The expression ``Independence Theorem" already appears in the earlier work on $S_{1}$-theories, and was borrowed from there. 
In addition to the original papers, there are several good texts on simple theories \cite{Wagner}, \cite{Kim-book}, \cite{Casanovas}.  The original definition of simplicity was in terms of not having the ``tree property". We will define it here in terms of ``dividing" as it is an opportunity to introduce dividing and forking.

\begin{Definition}\phantomsection\label{def first one}
\begin{enumerate}[(i)]
\item A formula $\phi(x,b)$ \emph{divides} over $A$ if there exists an $A$-indis\-cernible sequence $(b_{i}:i<\omega)$ with $b = b_{0}$ such that $\{\phi(x,b_{i}):i<\omega\}$ is inconsistent.
\item If $\Sigma(x)$ is a partial type over a set $B$ closed under conjunctions and $A\subseteq B$, then $\Sigma(x)$ \emph{divides} over $A$ if some formula $\phi(x,b)\in \Sigma(x)$  divides over $A$.
\item A formula \emph{forks} over $A$ if it implies a finite disjunction of formulas each of which divides over $A$.
\item For $\Sigma(x)$, $A\subseteq B$ as in (ii), $\Sigma(x)$ \emph{forks} over $A$ if some formula in $\Sigma(x)$ forks over $A$.
\item The complete theory $T$ is said to be {\em simple} if for any  complete type $p(x)\in S(B)$ there is a subset $A\subseteq B$ of cardinality at most $|T|$ such that $p(x)$ does not divide over $A$. 
\end{enumerate}
\end{Definition}

In simple theories, dividing and forking coincide. Stable theories can be characterized as simple theories such that for any model $M$, $p(x)\in S(M)$, and $M\prec N$,  $p(x)$ has a {\em unique} extension to a complete type $q(x)\in S(N)$ which does not fork over $M$. 

The \emph{stable forking conjecture} says that in a simple theory $T$, forking is explained by the ``stable part"  of $T$ (in a sense that we will not describe in detail).   There are many simple theories $T$ which have a stable reduct $T_{0}$ (with quantifier elimination) such that $T$ is the model companion of $T_{0}$  together with the new relations (possibly modulo some mild universal theory). Typically in such a situation forking in $T$ is witnessed by forking in $T_{0}$ so the stable forking conjecture holds. Our two main examples of simple theories will  have this feature. 

In a simple theory $T$ we will say that $a$ and $b$ are \emph{independent over $A$} (in the sense of nonforking) if $tp(a/A,b)$ does not fork over $A$.  This satisfies a number of properties: invariance, finite character, local character, 
existence of nonforking extensions, symmetry, transitivity, and the ``Independence Theorem over a model". Moreover the existence of an ``abstract independence relation" satisfying these properties implies simplicity of $T$ as well as that this relation coincides with nonforking.  This will be used in Sections \ref{sec 2} and \ref{sec 3} and we will give a few more details there.
Among the ``simplest" simple theories are the theories of $SU$-rank $1$, where every complete nonalgebraic $1$-type has only algebraic forking extensions. 

Although $NIP$ theories are {\em not} really objects of study in the current paper, they form part of the motivation. A theory $T$ is $NIP$ if there is no formula $\phi(x,y)$ and $a_{i}$ for $i\in \omega$ and $b_{S}$ for $S\subseteq \omega$ in some model $M$ of $T$ such that for all $i,S$, $M\models \phi(a_{i}, b_{S})$ iff $i\in S$.   $NIP$ theories are generalization of stable theories in an orthogonal direction from simple theories, and in fact $T$ is stable if and only if $T$ is both simple and $NIP$.  Although forking is not so well-behaved in NIP unstable theories, it still plays a big role. In particular, forking coincides with dividing over models \cite{Chernikov-Kaplan}, and global nonforking extensions of types over a model $M$ are precisely extensions which are invariant under automorphisms fixing $M$ pointwise.  For a type $p(x)$ over a set $A$ its global nonforking extensions (if they exist) are rather invariant over the bounded closure ``$bdd(A)$". 

The other main ingredients in this paper are \emph{Keisler measures}.  Given a structure $M$ (or model $M$ of $T$), and variable $x$, a Keisler measure $\mu_{x}$ over $M$ is a finitely additive probability measure on the Boolean algebra of definable (with parameters) subsets of the $x$-sort in $M$.  Keisler measures generalize complete types $p(x)$ over $M$ which are the special case where the measure is $\{0,1\}$-valued  ($0$ for false, $1$ for true).  It took a long time for Keisler measures to become part of everyday model theory (see \cite{Chernikov-MeasuresReview} for a quick survey).  They were studied by Keisler in \cite{Keisler} which is, on the face of it, about $NIP$ theories, but where, among the main points, is that for stable theories,  {\em locally} (formula-by-formula) Keisler measures are weighted, possibly infinite, sums of types.  (See also \cite{Pillay-domination} where this is used to give a pseudofinite account of the stable regularity lemma.)  In the $NIP$ environment, Keisler measures  were a very useful tool in solving some conjectures about definable groups in $o$-minimal structures \cite{NIPI}.  In \cite{NIPII}, \cite{NIPIII}, the ubiquity of automorphism (translation) invariant Keisler measures in $NIP$ theories (groups) was pointed out. In \cite{HKP} a first-order theory was defined to be \emph{amenable} if every complete type over $\emptyset$ extends to a global automorphism invariant Keisler measure. 

For pseudofinite fields,  the nonstandard counting measure provides both automorphism invariant measures on definable sets, as well as translation invariant measures on definable groups (with very good definability properties).   The examples given in Sections \ref{sec 2} and \ref{sec 3} of the current paper show in particular that such behaviour does not extend to simple theories in general.

We will now describe the main results of the paper,  with motivations coming  from what is known in the stable context. 

We will talk about  (non-)forking over $\emptyset$, but $\emptyset$ can be systematically replaced by any small set $A$ of parameters.

The following is well-known (\cite{NIPI}, \cite{Newelski-Petrykowski}) but we recall the proof anyway.
\begin{Fact}\label{fact forking implies zero} (No assumption on $T$.)  Suppose $\phi(x,b)$ forks over $\emptyset$. Then $\mu(\phi(x,b)) = 0$ for any  automorphism invariant global Keisler measure $\mu(x)$.
\end{Fact}
\begin{proof} Working in the saturated model $\bar M$ we may assume that $\phi(x,b)$ divides over $\emptyset$, witnessed by indiscernible sequence  $(b_{0}, b_{1}, \ldots)$ with $b_{0} = b$ such that 
$\{\phi(x,b_{i}):i<\omega\}$ is inconsistent.  So $\phi(x, b_{0})\wedge \phi(x,b_{1})\wedge \ldots \wedge \phi(x,b_{k})$ is inconsistent for some $k\geq 1$.
Assume for a contradiction that $\mu(\phi(x,b))  > 0$ for some automorphism invariant global Keisler measure $\mu$. Choose $0 \leq r <k$ maximum such that
$\mu(\phi(x,b_{0})\wedge \ldots \wedge \phi(x,b_{r})) = t$ for some $t>0$.  Let $\psi_{j}(x) =  \phi(x,b_{0})\wedge \ldots \wedge \phi(x,b_{r-1}) \wedge \phi(x,b_{j})$ for $j = r, r+1, r+2, \ldots$. Then by indiscernibility, invariance of $\mu$ and choice of $r$, we have that $\mu(\psi_{j}(x))) = t$ for all $j \geq r$, but  $\mu(\psi_{j}(x)\wedge \psi_{j'}(x)) = 0$ for $r\leq j < j'$ --- a contradiction as $\mu(x=x) = 1$. 
\end{proof} 

\begin{Remark}\label{rem stable global measure} Suppose $T$ is stable (and complete in language $L$), and $p(x)$ is a complete type over $\emptyset$. Then there is a global Keisler measure $\mu(x)$  (i.e. over a saturated model $\bar M$) which extends $p(x)$ and is $Aut(\bar M)$-invariant. Moreover $\mu$ is the {\em unique}  $Aut(\bar M)$-invariant global Keisler measure extending $p$.
\end{Remark}  
\begin{proof} Again we give a proof, for completeness.  The reader is referred to Section 2 of Chapter 1 of \cite{Pillay-GST}  for notation and facts that we use.  Fix a finite set $\Delta$ of $L$-formulas of the form $\phi(x,y)$, and consider the collection of $p'(x)|\Delta$ where $p'$ is a global nonforking extension of $p$. We know that there are only finitely many such,  say $p_{1},\ldots,p_{n}$.
Let $\mu_{\Delta}$ be the average of $\{p_{1}, \ldots ,p_{n}\}$, namely for each $\phi(x,y)\in \Delta$ and $b\in M$, $\mu_{\Delta}(\phi(x,b)) =  (1/n) (\sum p_{i}(\phi(x,b)))$ (where $p_{i}(\phi(x,b)) = 1$ if $\phi(x,b)\in p_{i}$ and $0$ otherwise).   

One has to check that $\Delta\subseteq \Delta'$ implies that $\mu_{\Delta'}$ agrees with $\mu_{\Delta}$ on $\Delta$-formulas, so that the directed union of the $\mu_{\Delta}$ gives a global Keisler measure $\mu$.  For this we use transitivity of the action of $Aut({\bar M})$ on the set of global nonforking extensions of $p$. From the definition of $\mu$ and invariance of non-forking, we deduce that $\mu$ is  $Aut(\bar M)$-invariant. 

Uniqueness of $\mu$ follows from  Fact \ref{fact forking implies zero}. 
\end{proof} 

\begin{Corollary}\label{cor nonfork positive measure} Suppose that $T$ is stable and $\phi(x,b)$ is a formula which does not fork over $\emptyset$. Then there is an $Aut(\bar M)$-invariant global Keisler measure giving $\phi(x,b)$ positive measure.
\end{Corollary}
\begin{proof} Let $p'$ be a global type which contains $\phi(x,b)$ and does not fork over $\emptyset$, and let $p$ be the restriction of $p'$ to $\emptyset$.  The $Aut(\bar M)$-invariant Keisler measure extending $p$ constructed in Remark \ref{rem stable global measure}  gives $\phi(x,b)$ positive measure.
\end{proof}

A weak version of the corollary above holds in $NIP$ theories using Proposition 4.7 of \cite{NIPII}.

The issue for the current paper is what happens in simple theories, where the role, if any, of Keisler measures was not well understood.   We will expand on some earlier comments. We fix a complete theory $T$, saturated model $\bar M$, sort $S$, and the Boolean algebra $\mathcal B$ of definable (with parameters) in $\bar M$ subsets of the sort $S$. The ideal $I_{f}$ is the collection of such definable sets which fork over $\emptyset$. For any $Aut(\bar M)$-invariant Keisler measure $\mu$ on $S$, let $I_{\mu}$ be the ideal of definable sets with $\mu$-measure $0$.  Fact \ref{fact forking implies zero} says that  $I_{f}\subseteq I_{\mu}$ for all such $\mu$.    In \cite{Hrushovski-approximate}, an $Aut(\bar M)$-invariant ideal $I$ of ${\mathcal B}$ was defined to be an $S_{1}$-ideal if for any $L$-formula $\phi(x,y)$ (where $x$ is of sort $S$) and indiscernible sequence $(b_{n}:n<\omega)$, if  $\phi(x,b_{1})\wedge \phi(x,b_{2}) \in I$, then $\phi(x,b_{1})\in I$.  Such $S1$ ideals appeared in the ``Stabilizer Theorem" from \cite{Hrushovski-approximate}.  Among the analogies between the forking ideal $I_{f}$ and the ideals $I_{\mu}$ is that (i)  $I_{\mu}$ is an $S_{1}$ ideal, and (ii) for simple $T$, $I_{f}$ is an $S_{1}$ ideal \cite{Kim-paper}.    The open problem (raised also by both  the first author and Leo Harrington in personal communications) is whether, in a simple theory $T$ (and working in a fixed sort), $I_{f}$ is the intersection of the $I_{\mu}$ for $\mu$ ranging over invariant global Keisler measures.  In the light of Fact \ref{fact forking implies zero}, this reduces to the question whether (in a simple theory) any formula (with parameters) which does not fork over $\emptyset$, has  $\mu$-measure $>0$ for some invariant measure $\mu$. Of course if there are no invariant Keisler measures on sort $S$, then the question has a negative answer, and Corollary \ref{cor simple and non amen} below gives such an example. However we are also interested in the situation where there do exist (many) invariant measures, namely where $T$ is also amenable in the sense described earlier.  So we prove:

\begin{Theorem}\label{thm: Main Thm 1} There is a simple theory $T$ (of $SU$-rank $1$) which  is amenable, together with a formula $\phi(x,b)$ which does not fork over $\emptyset$ but has measure $0$ for all
automorphism invariant global Keisler measures. 
\end{Theorem} 

We now turn to the case of definable groups. 
Recall:
\begin{Definition} Let $G$ be a group definable  (say without parameters) in a structure $M$. Then $G$ is said to be \emph{definably amenable} if there is a Keisler measure on $G$ over $M$ which is invariant under left translation by $G$.  
\end{Definition}

So definable amenability is a function not just of $(G,\cdot)$ but of the ambient structure $M$.
 
Recall from Section 5 of \cite{NIPI} that definable amenability of $G$ depends only on $Th(M)$, not the particular model chosen. The relation with paradoxical decompositions will be discussed in detail in Section \ref{sec 4}.  The group version of Remark \ref{rem stable global measure} is:
\begin{Fact}\label{fact stable gps are dfn amen}  Stable groups are definably amenable. More precisely if $Th(M)$ is stable and $G$ a group definable in $M$, 
then $G$ is definably amenable.  Moreover there is a {\em unique} left invariant Keisler measure on $G$ (over $M$) which is also the unique right invariant Keisler measure.
\end{Fact}

\begin{explanation}
This is well-known but spelled out in detail for the more general case of ``generically stable" groups in Corollary 6.10 of \cite{NIPII}. 
Also it is done explicitly in the local (formula-by-formula) case in \cite{CPT1}.
\end{explanation}

It was asked by several people, including the sixth  author, whether groups definable in models of simple theories are definably amenable. Note that this {\em is} the case for groups definable in pseudofinite fields (or arbitrary pseudofinite theories). Nevertheless, our second main result is:
\begin{Theorem}\label{thm Main theorem 2} There is a simple theory (of $SU$-rank $1$) and a definable group $G$ in it which is NOT definably amenable. 
\end{Theorem}

The usual move of expanding a theory by a new sort for a principal homogeneous space ($PHS$) for a definable group yields:
 \begin{Corollary}\label{cor simple and non amen} There is a simple theory which is NOT amenable.  In fact there is a sort $S$ with a unique $1$-type over $\emptyset$, such that there is no global invariant Keisler measure on sort $S$. 
\end{Corollary}

We recall briefly the situation for definable groups in $NIP$ theories.  
First there DO exist non definably amenable groups; such as $SL(2,\R)$ as a group definable in the real field.
Nevertheless there is a very nice theory of definably amenable groups, beginning in \cite{NIPII}, continued in \cite{NIPIII,Chernikov-Pillay-Simon} and brought to a fairly comprehensive conclusion in 
\cite{Chernikov-Simon}.  The latter paper includes a {\em classification} of the translation invariant Keisler measures on definably amenable groups in $NIP$ theories. 

Theorem \ref{thm: Main Thm 1} will be proved in Section \ref{sec 2}.  Theorem \ref{thm Main theorem 2} and Corollary \ref{cor simple and non amen} will be proved in Section \ref{sec 3}.
The constructions of the theories and structures which give these (counter-)examples are a bit complicated from the combinatorial point of view.

There is a general theory of ``definable paradoxical decompositions" from \cite{NIPI}, which gives obstructions to definable amenability of groups.  A general problem is to determine which interesting model-theoretic properties are inconsistent with the existence of a definable paradoxical decomposition.  In Section \ref{sec 4}, we show directly that smallness of $T$ (as well as stability)  is such a property, yielding the definable amenability of groups definable in small theories and in stable theories (although the latter was given earlier in the paper).  We also give a ``simpler" witness to  Theorem \ref{thm Main theorem 2}, in terms of certain invariants related to definable paradoxical decompositions.  Finally we discuss Grothendieck rings of structures, and show the non-triviality of the {\em graded Grothendieck ring} of any structure with small theory. 

\smallskip
Thanks to the referee for many helpful suggestions.

\section{A simple theory where forking is not detected by measures}\label{sec 2}

Here we prove Theorem \ref{thm: Main Thm 1}. 
We first give an overview and then the technical details. 
Recall first that for any group $G$ and a free action of $G$ on a set $P$ we can consider $P$ as a structure in a language with function symbols $f_{g}$ for each $g\in G$.
When $G$ is infinite, all such structures are elementarily equivalent, the theory is strongly minimal and there is a unique $1$-type over $\emptyset$. We will choose $G$ to be the free group $F_5$ on $5$ generators. 
We will add another sort $O$ to the picture and a relation $R\subseteq O\times P$ and find $a_{1}, \ldots ,a_{5}$ in $P$ such that  $R(x,a_{1})$, $R(x,a_{2})$, $R(x,a_{3})$ are disjoint infinite sets, which are contained in the union of $R(x,a_{4})$ and $R(x,a_{5})$.  It will be done sufficiently generically such that there is still a unique $1$-type realized in $P$, and the theory of the structure is simple (of $SU$-rank $1$).  As all of the $a_{i}$  have the same type, any automorphism invariant Keisler measure  (on the sort $O$) will assign the same measure to each of the $R(x,a_{i})$, which will have to be $0$. But $R(x,a_{i})$  (being infinite) does not fork over $\emptyset$. 

\subsection{The universal theory}\label{subs 2.1}
As usual we mix up notation for symbols of the language and their interpretations. 
As above we have two sorts $O$, $P$, and relation $R\subseteq O\times P$.  And it is convenient to only have function symbols for $5$ free generators of $F_{5}$ and their inverses, which we will call $f_{1}^{\pm}$, $f_{2}^{\pm}$,  $f_{3}^{\pm}$, $g_{1}^{\pm}$, and $g_{2}^{\pm}$. We get a language $L$.  Terms corresponds to elements of the free group $F_{5}$, which will act on the sort $P$, via the function symbols. For $a\in P$, let $R_{a}$ denote the subset of $O$ defined by $R(x,a)$.

Then we can express by a collection of universal sentences in $L$ that
\begin{enumerate}[(i)]
\item the map taking $(t,a)\in G\times P$ to $ta\in P$ is a free action of $G$ on $P$,
\item  for all $a\in P$,  the sets (subsets of $O$), $R_{f_{1}(a)}$, $R_{f_{2}(a)}$,  $R_{f_{3}(a)}$ are pairwise disjoint and each is contained in the union of $R_{g_{1}(a)}$ and $R_{g_{2}(a)}$. 
\end{enumerate}
We will call this universal $L$-theory $T$. 

We will  define a theory $T^{*}$ in $L$ which extends $T$ and has quantifier elimination, so will be the model companion of $T$. As usual to show the existence of model companions
one needs to describe, in the parameters, when a quantifier-free formula $\phi(x)$ over a model $M$ of $T$ has a solution in a larger model $N$ of $T$. The key issue is Axiom (ii) above. 
 So some combinatorics is required which will be done in the next section.

\subsection{Colourings and free actions}\label{subs 2.2}
We fix a free action of $F_{5}$ on a set $X$.  
As above, we will denote by $\{f_{1}, f_{2}, f_{3}, g_{1}, g_{2}\}$ a system of free generators for the free group $F_{5}$.  
 There is an induced graph structure on $X$, where we put an edge between $u$ and $v$ if $v = gu$ for $g$ one of the distinguished generators $f_{i}, g_{j}$ or its inverse.  If $u, v\in X$ are distinct, then by a {\em path} between $u$ and $v$, we mean a sequence $u_{0},..,u_{n}$ of distinct elements of $X$ such that $u_{0} = u$, $u_{n} = v$, and $u_{i}, u_{i+1}$ are joined by an edge for $i=0,..,n-1$.  As the action of $F_{5}$ on $X$ is free, there is at most one path between distinct elements $u,v$ of $X$ and we have the corresponding metric $d$. $d(u,v) = 0$ if $u=v$, and is the length (number of edges) of the path between $u,v$ if there is such a path and $= \infty$ otherwise.  So if $d(u,v) = n>0$ it means that there is a (unique) reduced word $w$ of length $n$ in the generators and their inverses such that $wu = v$.  In the case that $u=v$, it is convenient to define $\{u\}$ to be the path between $u$ and $v$, which is of length $0$. 

We will have a similar set-up in Section \ref{subs 3.2} but with $F_{12}$ in place of $F_{5}$.


For $v\in X$, let $B_{n}(v)$, the ball around $v$ of radius $n$, be $\{u\in X: d(v,u) \leq n\}$ and for $V$ a subset of $X$, $B_{n}(V) = \bigcup_{v\in V}B_{n}(v)$.

\begin{Definition}\phantomsection\label{def 2.1}
\begin{enumerate}[(i)]
\item Define $\leq^{*}$ on $X$ by  $u\leq^{*}v$  if there exist $i\in [3]$ and $j\in [2]$ such that $v = g_{j}f_{i}^{-1}u$.
\item Let $\leq$ be the reflexive and transitive closure of $\leq^{*}$, and for $v\in X$, let $U_{v} = \{u\in X: v\leq u\}$. 
\item  The $nth$ level of $U_{v}$ is $\{u\in U_{v}: d(v,u) = 2n\}$. 
\item By a complete tree for $v\in X$ we mean a subset $T$ of $X$ containing $v$ such that for all $u\in T$, and $i\in [3]$ there is $j\in [2]$ such that $g_{j}f_{i}^{-1}(u)\in T$.
\item By a depth $n$ tree for $v\in X$, we restrict (iv) to $T\subseteq B_{2n}(v)$ and require the second clause of (iv) only for $u\in B_{2n-2}(v) \cap T$.
\end{enumerate}
\end{Definition}

\begin{Remark}\phantomsection\label{rem colors trees} 
\begin{enumerate}[(a)]
\item  Explanation of (v): Note that if  $d(v,u) = 2n-2$ then for any $i\in [3]$ and  $j\in [2]$, $g_{j}f_{i}^{-1}u$ has distance at most $2n$ from $v$. 
\item  Any product of words of the form $g_{j}f_{i}^{-1}$  for $i\in [3]$ and $j\in [2]$ will be a reduced word. Hence if $w, w'$ are distinct such reduced words, and $u,v\in X$ then we could not have that both $wu = v$ and $w'u = v$. 
\end{enumerate}
\end{Remark}

\begin{Lemma}\label{lem complete tree disjoint from Y} Suppose $v\in X$, and $Y\subset X$ with $|Y| \leq  n+1$. Suppose there is a depth  $n$ tree $T$ for $v$ with $T\cap Y= \emptyset$. Then there is a complete tree $T'$ for $v$  which is disjoint from $Y$.
\end{Lemma}
\begin{proof} The proof is by induction on $n$.  When $n = 0$, we may assume $Y$ is a singleton $\{x\}$, and $T = \{v\}$ with $v\neq x$.

For $i\in [3]$ and $j\in [2]$ let $v_{i,j} =  g_{j}f_{i}^{-1}v$.   By Remark \ref{rem colors trees}(b), there will be at most one $v_{i,j}$ such that  $v_{i,j}\leq x$.   Hence for each $i\in [3]$ there is $j(i)\in [2]$ such that  $v_{i,j(i)} \nleq x$. Hence also for each $i\in [3]$,  $x\notin U_{v_{i,j(i)}}$.  Hence  $\{v\}\cup \bigcup_{i\in [3]}U_{v_{i,j(i)}}$ is a complete tree for $v$ which is disjoint from $Y = \{x\}$. 

The inductive step:  Suppose $|Y| = n+1$ and $T$ is a depth $n$ tree for $v$ such that  $T\cap Y = \emptyset$  (and $n>0$).  As above denote by $v_{i,j}$, $g_{j}f_{i}^{-1}v$.  Fix $i\in [3]$ and one of the $j$'s $\in [2]$ such that $v_{i,j}\in T$.  Then clearly $T\cap U_{v_{i,j}}$ is a depth $n-1$ tree for $v_{i,j}$ which is disjoint from $Y$. 
\newline
{\em Case 1.} $|Y\cap U_{v_{i,j}}| \leq n$. Then by induction hypothesis, there is complete tree $T_{i}$ for  $v_{i,j}$ which is disjoint from $Y\cap U_{v_{i,j}}$. As $T_{i}\subseteq U_{v_{i,j}}$ it follows that $T_{i}$ is also disjoint from $Y$.
\newline
{\em Case 2.}   $|Y\cap U_{v_{i,j}}| = n+1$.  Namely $Y\subseteq  U_{v_{i,j}}$.  Let $j'\neq j$, $j'\in [2]$. So clearly  $U_{v_{i,j'}}$ is disjoint from $U_{v_{i,j}}$ (again by freeness of the action of $F_{5}$) and so disjoint from $Y$. In this case  define $T_{i}$ to be  $U_{v_{i,j'}}$, a complete tree for $v_{i,j'}$ which is disjoint from $Y$. 

Now let $T' = \{v\} \cup \bigcup_{i\in [3]}T_{i}$.  Then $T$ is disjoint from $Y$ and is a complete tree for $v$. 
\end{proof} 

The motivation  for part (1) of the next definition  is  to use colourings to describe quantifier-free $1$ types over $P$ realized in $O$ in models of $T$.  That is, a colouring $c$ of $P$ with colours $+,-$ will correspond to the quantifier-free type $p(x)$ on $O$ where $R(x,a)\in p(x)$ iff $c(a) = +$.  Conditions (a) and (b) below correspond to Axiom (ii) from the universal theory $T$. 

\begin{Definition}\phantomsection\label{def good color}
\begin{enumerate}[(1)]
\item  Suppose $D\subseteq X$.  By a {\em good colouring} of $D$ we mean a function $c:D \to \{+,-\}$,  such that if $v\in D$ and $c(v) = +$ then
\begin{enumerate}[(a)]
\item for all $i\in [3]$ there is $j\in [2]$ such that $c(g_{j}f_{i}^{-1}(v)) = +$  if $g_{j}f_{i}^{-1}(v) \in D$.
\item and for all $i\neq j \in [3]$, $c(f_{j}f_{i}^{-1}v) = -$, if $f_{j}f_{i}^{-1}v \in D$.
\end{enumerate}
Moreover if $D = X$ we call $c$ a {\em total} good colouring.
\item We say that $v_{1}, v_{2}\in X$ are a {\em conflicting pair}, if there are $w_{1}\in U_{v_{1}}$ and $w_{2}\in U_{v_{2}}$ such that  $w_{2}= f_{j}f_{i}^{-1}w_{1}$ for some $i\neq j \in [3]$. 
\end{enumerate}
\end{Definition}

\begin{Lemma}\phantomsection\label{lem conflicting pair} 
\begin{enumerate}[(i)]
\item  Being a conflicting pair is symmetric.
\item If $v_{1}$ and $v_{2}$ are a conflicting pair, then  there are {\em unique} $w_{1}\in U_{v_{1}}$ and $w_{2}\in U_{v_{2}}$ such that  $w_{2} = f_{j}f_{i}^{-1}w_{1}$ for some $i\neq j\in [3]$.  We call $w_{1}, w_{2}$ {\em the}  conflict points.
\end{enumerate}
\end{Lemma}
\begin{proof}  (i) is obvious.
\newline
(ii)  Let $w_{1}\in U_{v_{1}}$, $w_{2}\in U_{v_{2}}$ witness that $v_{1}$ and $v_{2}$ are a conflicting pair, namely  $w_{2} = f_{j}f_{i}^{-1}w_{1}$ for some $i\neq j \in [3]$.  Let  $w_{1} = 
xv_{1}$   and $w_{2} = yv_{2}$,  where $x$ and $y$ are products (maybe empty) of pairs of free generators of the form $g_{k}f_{\ell}^{-1}$ (as $w_{1}\in U_{v_{1}}$ and $w_{2}\in U_{v_{2}}$).  Then $v_{2} = y^{-1}f_{j}f_{i}^{-1}xv_{1}$. The product  $y^{-1}f_{j}f_{i}^{-1}x$ is already reduced (as $y^{-1}$ ends and $x$ begins with a $g$-generator).  Thus $x$ and $y$ are uniquely determined, hence $w_{1}$ and $w_{2}$ too. 
\end{proof}

\begin{Proposition}\label{prop extending color} Let $V$ and $W$ be disjoint finite subsets of $X$, both of which have cardinality at most $n$. Let $c: V\cup W \to \{+,-\}$ be a good colouring of  $V\cup W$ given by $c$ is $+$ on $V$ and $-$ on $W$.  Let $N = n(n+1) -2$.  Then there is total  good colouring (i.e. of $X$) extending $c$ if and only if there is good colouring of $B_{N}(V)$ extending the restriction of $c$ to  $B_{N}(V) \cap (V\cup W)$.
\end{Proposition}
\begin{proof}   One direction is obvious: if $c'$ is a total good colouring then its restriction to $B_{N}(V)$ of course extends its further restriction to $B_{N}(V)\cap (V\cup W)$. 

For the other direction: suppose $c'$ is a good colouring of $B_{N}(V)$ extending the restriction of $c$ to $B_{N}(V)\cap (V\cup W)$.

Note in passing that $V\subseteq B_{N}(V)$.  We will define a set $Y$ which consists of $W$ together with one element from each pair $(w,w')$ of conflict points which come from a conflicting pair $(v_{1}, v_{2})$ of elements of $V$.  So given such $v_{1}, v_{2}\in V$ and conflict points $w_{1}, w_{2}$:

\noindent{\em Case 1.} Both $w_{1}, w_{2}\in B_{N}(V)$.   Then by the good colouring condition 1(b) (from Definition \ref{def good color}), not both $c'(w_{1})$ and $c'(w_{2})$ equal $+$. So choose one of them, without loss $w_{1}$ such that $c'(w_{1}) = -$ and put $w_{1}$ into $Y$.

\noindent{\em  Case 2.}  At least one of $w_{1}, w_{2}$, without loss $w_{1}$ is NOT in $B_{N}(V)$. Then add $w_{1}$ to $Y$.  

There are at most $n(n-1)/2$  conflicting (unordered) pairs from $V$, and hence  $|Y| \leq  n   +   n(n-1)/2  = n(n+1)/2 =  N/2 + 1$, and by construction $c'(x) = -$ for all $x\in Y\cap B_{N}(V)$. 

Now for each $v\in V$, $T = \{u\in B_{N}(v): c'(u) = +\}$ is a depth $N/2$ tree for $v$ which is disjoint from $Y$ (by definition of a good colouring and the construction of $Y$). 
By Lemma \ref{lem complete tree disjoint from Y} (as $|Y|
\leq N/2  + 1$)  there is, for each $v\in V$, a complete tree $T_{v}$ for $v$ which is disjoint from $Y$. Let us then define a (total) colouring $c''$ of $X$ which has value $+$ on $T_{v}$ for each $v\in V$ and $-$ otherwise. 

As $c$ is $+$ on $V$, and $-$ on $W$ which is contained in $Y$ which is disjoint from each $T_{v}$, $c''$ extends $c$.

\begin{Claim*}   
$c''$ is good. 
\end{Claim*}
\noindent{\em Proof of Claim.} Suppose $c''(u) = +$. So $u\in T_{v}$ for some $v\in V$.  But $T_{v}$ is a complete tree for $v$, so for each $i\in [3]$ there is $j\in [2]$ such that $g_{j}f_{i}^{-1}u\in T_{v}$, whereby $c''(g_{j}f_{i}^{-1}u) = +$.  This gives 1(a) in the definition (Definition \ref{def good color}) of a good colouring.

For 1(b): suppose for a contradiction that $c''(w_{1}) = +$ and $c''(w_2) = +$ for $w_{1}$, $w_{2}$ in $X$ such that $w_{2} = f_{j}f_{i}^{-1}w_{1}$ for some $i\neq j\in [3]$. 
But then $w_{1}\in T_{v_{1}}$ and $w_{2}\in T_{v_{2}}$ for some $v_{1}, v_{2}\in V$, and we see that $w_{1}, w_{2}$ are conflict points for the conflicting pair $v_{1}, v_{2} \in V$.
But by the definition of $Y$, one of $w_{1}, w_{2}$ is in $Y$ and so gets $c''$ colour $-$. A contradiction. 
\end{proof}

\begin{Corollary} For each $v\in X$ there are good colourings $c,c'$ of $X$ such that $c(v) = +$ and $c'(v) = -$. 
\end{Corollary}

\subsection{The model companion $T^{*}$}\label{subs 2.3}
We return to the context of Section \ref{subs 2.1}, namely the language $L$ and universal theory $T$. To any element $h$ of $F_{5}$ expressed in terms of the generators and their inverses in reduced form we have a term $t_{h}$ of $L$. Note that if $t$ is a term in nonreduced form then there will be some $h$ such that $t = t_h$ is true in all models of $T$. 

We will give two axiom schema, which in addition to $T$ give a theory $T^{*}$ in the given language.
We will check subsequently that $(T^{*})_{\forall} = T$, and that $T^{*}$ has quantifier elimination (and is complete), so is the model companion of $T$. 

We want to describe which quantifier-free $1$ types over a model $M$ of $T$ can be realized in some extension $N$ of $M$ to a model of $T$, by expressing the existence of solutions of appropriate approximations. There are two kinds of $1$-types: realized by an element of $P$, and realized by an element of $O$. We introduce some notation to  deal with each of these cases. 

Let $p_{i}(z,x)$ for $i\in I$ be a list of all (complete)  quantifier-free types (over $\emptyset$) of pairs $(a,b)$ in models $M$ of $T$ where $a\in O(M)$ and $b\in P(M)$.  So $p_{i}(z,x)$ will be
a maximal consistent (with $T$) set of formulas of the form $R(z,t_{h}(x))$, $\neg R(z,t_{h}(x))$  for $h$ ranging over $F_{5}$.  (The inequalities between $x$ and the $t_{h}(x)$ for $h\neq 1$ will come free from $T$).

 For each $n$, let $\gamma_{n}(x_{1}, \ldots ,x_{n},y_{1},\ldots,y_{n})$ be a quantifier-free $L$-formula expressing the existence of a good colouring $c$ of $B_{N}(\{x_{1},\ldots,x_{n}\})$ such that  $c(x_{i}) = +$ for $i=1,\ldots,n$ and $c(y_{i}) = - $ for  each $y_{i}$ which happens to be in $B_{N}(\{x_{1},\ldots,x_{n}\})$ (where $N = n(n+1) - 2$).

\begin{Axiom}\label{axiom I}
All sentences of the form  
\begin{gather*}
(\forall x_{1},\ldots,x_{n} \in P) (\forall z_{1},\ldots,z_{n}\in O)\Big(\bigwedge_{i\neq j}z_{i}\neq z_{j} \to \\
(\exists x\in P) \big(\bigwedge_{j=1,\ldots,n}\phi_{i_{j}}(z_{j},x) \wedge \bigwedge_{i=1,\ldots,n} x\neq x_{i} \big) \Big),
\end{gather*}
where $n\geq 1$,  $i_{1}, \ldots, i_{n}\in I$ and each $\phi_{i_{j}}(z,x)$ is a finite conjunction of formulas in $p_{i_{j}}(z,x)$. 
\end{Axiom}

\begin{Axiom}\label{axiom II}
All sentences of the form
\begin{gather*}
	(\forall x_{1},\ldots,x_{n},y_{1},\ldots,y_{n}\in P)(\forall z_{1},\ldots,z_{n}\in O)\Big(\gamma_{n}(x_{1},\ldots,x_{n},y_{1}, \ldots, y_{n}) \to \\
	 (\exists z\in O) \big(\bigwedge_{i=1,\ldots,n}(R(z,x_{i})\wedge \neg R(z,y_{i})) \wedge \bigwedge _{i=1,\ldots,n}z\neq z_{i} \big) \Big)
\end{gather*}
for $n \geq 1$.
\end{Axiom}

We define $T^{*}$ to be (the theory axiomatized by) $T$ together with Axiom Schemas \ref{axiom I} and \ref{axiom II}.

\begin{Lemma}\label{lem ecm of T is of T*} Any existentially closed model of $T$ is a model of $T^{*}$. In particular $T^{*}$ is consistent and $(T^{*})_{\forall} = T$. 
\end{Lemma}
\begin{proof} 
Let $M$ be an existentially closed model of $T$.  Consider an axiom
\begin{gather*}
	(\forall x_{1},\ldots,x_{n} \in P) (\forall z_{1},\ldots,z_{n}\in O)\Big(\bigwedge_{i\neq j}z_{i}\neq z_{j} \to\\
	 (\exists x\in P) \big( \bigwedge_{j=1,\ldots,n}\phi_{i_{j}}(z_{j},x) \wedge \bigwedge_{i=1,\ldots,n} x\neq x_{i} \big) \Big)
\end{gather*}
belonging to Axiom Schema \ref{axiom I}.   

Choose $a_{1},\ldots,a_{n}\in O(M)$, which we may assume to be distinct.  We will build a certain model $M'$ of $T$ containing $M$. Let  $X$ be a principal homogeneous space for $F_{5}$ (disjoint from $P(M)$) with a distinguished point $b$. Let $P(M') = P(M) \cup X$ with the natural action of $F_{5}$. For $h \in F_5$, we put $(a_{j},hb)\in R$ iff  $R(z,t_{h}(x))\in p_{i_{j}}(z,x)$. And for any other $a\in O(M)$, we put $\neg R(a,c)$ for any $c\in X$.   We also define $O(M')$ to be $O(M)$.
Then it can be checked that $M'$ is a model of $T$.
Now $b$ witnesses that  the formula  $(\exists x\in P)(\bigwedge_{j=1,\ldots,n}\phi_{i_{j}}(a_{j},x) \wedge \bigwedge_{i=1,\ldots,n} x\neq b_{i}))$  for any  $b_{1},\ldots,b_{n}\in P(M)$ holds in $M'$. As $M$ is existentially closed in $M'$, this formula also holds in $M$.  We have shown that $M$ is a model of Axiom Schema \ref{axiom I}. 

Now let  
\begin{gather*}
	(\forall x_{1},\ldots,x_{n},y_{1},\ldots,y_{n}\in P)(\forall z_{1},\ldots,z_{n}\in O) \Big( \gamma_{n}(x_{1},\ldots,x_{n},y_{1}, \ldots ,y_{n}) \to \\
	 (\exists z\in O) \big( \bigwedge_{i=1,\ldots,n}(R(z,x_{i})\wedge \neg R(z,y_{i})) \wedge \bigwedge _{i=1,\ldots,n}z\neq z_{i} \big) \Big)
\end{gather*}
be a sentence in Axiom Schema \ref{axiom II}. 

Choose $b_{1}, \ldots, b_{n}, c_{1},\ldots,c_{n} \in P(M)$.  We will add a new point $\star$ to the $O$ sort to get a structure $M'$ extending $M$.  Let us assume that $M\models \gamma_{n}(b_{1},\ldots,b_{n}, c_{1},\ldots,c_{n})$. By Proposition \ref{prop extending color}, there is a good colouring $c$ of $P(M)$ such that $c(b_{i}) = +$ and $c(c_{i}) = -$ for $i=1, \ldots, n$.  
For $d\in P(M) = P(M')$ we define $R(\star, d)$  iff  $c(d) = +$.  Then $M'$ is a model of $T$, and again as $M$ is existentially closed in $M'$,  $(\exists z\in O)(\bigwedge_{i=1,\ldots,n}(R(z,b_{i})\wedge \neg R(z,c_{i})) \wedge \bigwedge _{i=1,\ldots,n}z\neq a_{i}))$  is true in $M$, for any $a_{1},\ldots,a_{n}\in O(M)$.
So $M$ is a model of Axiom Schema \ref{axiom II}. 
\end{proof} 

\begin{Proposition}\phantomsection\label{prop: Tstar in Sec 2} 
\begin{enumerate}[(i)]
\item $T^{*}$ is complete with quantifier elimination,
\item $T^{*}$ is the model companion of $T$,
\item for any model $M$ of $T^{*}$  and $A\subseteq M$, the algebraic closure of $A$ in $M$ (in the sense of the structure $M$) is precisely $\langle A \rangle$, the substructure of $M$ generated by $A$.
\end{enumerate}
\end{Proposition}
\begin{proof} For (i) we use the well-known criterion that for $M, N$  $\omega$-saturated models of $T^{*}$, the collection of partial isomorphisms
between finitely generated substructures of $M$ and $N$ is nonempty and has the back-and-forth property. 

First to show nonemptiness: Let $a\in O(M)$ and $b\in O(N)$. Then $\{a\}$, $\{b\}$ are isomorphic substructures of $M$ and $N$.

Now suppose $f$ is an isomorphism  between finitely generated substructures $M_{0}$ and $N_{0}$ of $M$ and $N$ respectively.
Let $a\in M$. We want to extend $f$ to $g$ with $a\in dom(g)$. We may assume $a\notin M_{0}$. 

\noindent
{\em Case 1.}  $a\in P(M)$.

\noindent Let $p(x) = qftp(a/M_{0})$  (quantifier-free type of $a$ over $M_{0}$).  For each $b\in O(M_0)$, let $p_{b}(z,x) = qftp(b,a/\emptyset)$. 
Then $p(x)$ is axiomatized by  $\{x\neq c: c\in P(M_0)\} \cup \bigcup_{b\in O(M_0)}p_{b}(b,x)$.
Now $f(p)$ is precisely  $\{x\neq d: d\in P(N_0)\} \cup \bigcup_{b\in O(M_0)}p_{b}(f(b),x)$. 

By Axiom Schema \ref{axiom I} and $\omega$-saturation, $f(p)$ is realized in $N$.  

\noindent
{\em Case 2.}  $a\in O(M)$. 

\noindent Let $q(z) = qftp(a/M_{0})$.  Then $f(q) = \{z\neq d: d\in O(N_0)\} \cup  \{R(z,f(b)):  b\in P(M_0)$, $M\models R(a,b)\} \cup \{\neg R(z,f(b)): b\in P(M_0)$, $M\models \neg R(a,b)\}$. 
Choose $b_{1}, \ldots,b_{n}\in P(M_0)$ such that $M\models R(a,b_{i})$, and  $c_{1},\ldots,c_{n}\in P(M_0)$ such that $M\models \neg R(a,c_{i})$  (if such exist).
Then as $M$ is a model of $T^{*}$ (and so of $T$) we have $M\models \gamma_{n}(b_{1},\ldots,b_{n},c_{1},\ldots,c_{n})$, whereby  
\[N\models \gamma_{n}(f(b_{1}),\ldots,f(b_{n}), f(c_{1}),\ldots,f(c_{n})).\]
So by Axiom Schema \ref{axiom II} and the $\omega$-saturation of $N$, $f(q)$ is realized in $N$.

\noindent
(ii) follows immediately as $T^{*}$ is model-complete (by (i)) and  $(T^{*})_{\forall} = T$ (by Lemma \ref{lem ecm of T is of T*}).

\noindent
(iii)  By quantifier-elimination,    we have to show that for any small substructure $M_{0}$ of a (saturated) model of $T^{*}$, and $a\in M\setminus M_{0}$, $qftp(a/M_{0})$ has infinitely many realizations.  For $a\in P(M)$ this is by Axiom Schema \ref{axiom I} and saturation. And for $a\in O(M)$ this is by Axiom Schema \ref{axiom II} and saturation. 
\end{proof}

\subsection{Simplicity and the proof of Theorem \ref{thm: Main Thm 1}}\label{subs 2.4}

We now work in a saturated model $\bar M$ of the complete theory $T^{*}$ defined earlier. 

\begin{Proposition}\label{prop: div in Sec 2} Let $a$ be an element (so an element of  $O(\bar M)$ or of $P(\bar M)$), and $B$ a (small) subset.  Then $a\notin acl(B)$ implies that $tp(a/B)$ does not divide over $\emptyset$.
\end{Proposition} 
\begin{proof}
We may assume that $B$ is a substructure, enumerated by an infinite tuple $b_{0}$.  Let $I = (b_{0}, b_{1}, b_{2}, \ldots)$ be an indiscernible sequence. Note that $\bigcup I$ is a substructure, say $M_{0}$, of $\bar M$. 

Let $p(x,b_{0}) = tp(a/b_{0})$ with $a \notin B$. 

\noindent
{\em Case 1.}  $a\in P(\bar M)$. 

\noindent Define a new structure $M_{1}$ extending $M_0$, by adjoining new elements $\{\star_{g}:g\in F_{5}\}$ satisfying $P$, and for any element $c$  in some $b_{n}$  such that $O(c)$, define $R$ to hold of  $(c,\star_{g})$ iff  the corresponding element of  $b_{0}$ is in the relation $R$ with $t_{g}(a)$. Also define the $f_{i}^{\pm}$ and $g_{j}^{\pm}$ tautologically on $\{\star_{g}:g\in F_{5}\}$. 
Then check that $M_{1}$ is a model of $T$, so by quantifier elimination and saturation of $\bar M$ we may assume that $M_{1}$ is an extension of $M_{0}$ inside $\bar M$. And we see that 
$\star_{e}$ realizes $p(x,b_{i})$ for all $i$.  

Hence $p(x,b_{0})$ does not divide over $\emptyset$.

\noindent
{\em Case 2.} $a\in O(\bar M)$. 

\noindent Do the analogous thing: define an $L$-structure extending $M_{0}$ with a single new element $\star$ which is in $O$ and with $R(\star,c)$ for $c$ in some $b_{n}$ (such that $P(c)$)  iff $a$ is $R$-related to the corresponding element of $b_{0}$.  Again check that we get a model of $T$, so can be assumed to live in $\bar M$ over $M_{0}$ and $\star$ realizes $p(x,b_{i})$ for all i.
\end{proof} 

\begin{Corollary}
\begin{enumerate}[(i)]
\item  $T^{*}$ is simple and of $SU$-rank $1$ (each of the sorts $O$, $P$ has $SU$-rank $1$). 
\item For all tuples $a$, $b$ and  subset $A$ (of $\bar M$), $a$ is independent from $b$ over $A$ iff  $\langle aA \rangle \cap \langle bA \rangle =  \langle A \rangle$. 
\item Each of the sorts has a unique $1$-type over $\emptyset$.
\end{enumerate} 
\end{Corollary}
\begin{proof}  By Proposition \ref{prop: div in Sec 2}, every complete $1$-type (over any set) is either algebraic or does not divide over $\emptyset$, which implies that $T$ is simple. 
In particular forking equals dividing and is symmetric.  And so the proposition says that the only forking extensions of any complete $1$-type are algebraic, namely that each of the sorts has $SU$-rank $1$.

(ii) follows from  Proposition \ref{prop: div in Sec 2} (and Proposition \ref{prop: Tstar in Sec 2} (iii)) by  forking calculus, using also the fact for any set $B$, $ \langle B \rangle = \bigcup_{b\in B} \langle b \rangle$, which follows from there being only unary function symbols in the language. 

And (iii) is a consequence of quantifier elimination. 
\end{proof}

The proof of Theorem \ref{thm: Main Thm 1} is completed by the following results:
\begin{Proposition} For any $a\in P$, the formula $R(z,a)$ does not fork over $\emptyset$ but has measure $0$ for any (automorphism) invariant Keisler measure $\mu$ (on the sort $O$). 
\end{Proposition}
\begin{proof} Let $\mu$ be an invariant Keisler measure on the sort $O$.  As there is a unique $1$-type over $\emptyset$ realized in $P$, $\mu(R(x,a)) = \mu(R(x,b))$ for all $a,b\in P$. 
But for any given $a$, and $i\in [3]$, $R(x,f_{i}(a)) \to \left( R(x,g_{1}(a)) \vee R(x,g_{2}(a)) \right)$, and $R(x,f_{1}(a))$, $R(x,f_{2}(a))$, $R(x,f_{3}(a))$ are pairwise inconsistent.  So this forces 
$\mu(R(x,a)) = 0$ for all $a\in P$.  On the other hand $R(x,a)$ has infinitely many realizations, so as $O$ has $SU$-rank $1$, $R(x,a)$ does not fork over $\emptyset$. 
\end{proof} 

\begin{Proposition} The theory $T^{*}$ is \emph{extremely amenable}: every complete type over $\emptyset$ has a global (automorphism) invariant extension.
\end{Proposition}
\begin{proof} We just give a sketch, leaving details to the interested reader.
Let $p({\bar x}, {\bar z}) = tp({\bar a}, {\bar b}/\emptyset)$ where ${\bar a}$ is a tuple from $P$ and ${\bar b}$ a tuple from $O$.
Let $M$ be a saturated model. Then we can find a realization $({\bar a}', {\bar b}')$ of $p$ in some elementary extension $N$ of $M$ such that all the elements from the tuple $({\bar a}', 
{\bar b}')$ are in $N\setminus M$, $N\models \neg R(d, a)$ for each $a\in {\bar a}'$, $d\in O(M)$,  and $N\models \neg R(b,d)$ for each $b\in {\bar b}'$ and $d\in P(M)$. 
Then $tp(({\bar a}', {\bar b}')/M)$ is clearly $Aut(M)$-invariant  (using quantifier elimination). 
\end{proof}

\section{A non definably amenable group definable in a simple theory}\label{sec 3}
In this section we will prove Theorem \ref{thm Main theorem 2}.  Again we start with  an overview.  Our theory $T^{*}$ will be a certain  expansion of $ACF_{0}$, and the group $G$ which is not definably amenable will be $SL_{2}(K)$, where $K$ is the underlying algebraically closed field.  Of course working just in $ACF_{0}$, $SL_{2}(K)$ will be definably (extremely) amenable. The additional structure we will add will be a partition of $SL_{2}(K)$ into $4$ sets $C_{1}, C_{2}, C_{3}, C_{4}$.  We will choose matrices  $a(i,j)\in SL_{2}(\Z)$ for $i\in [4]$, $j\in [3]$, which freely generate  $F_{12}$, and require that for each $i\in [4]$, $\bigcup_{j\in [3]} a(i,j)^{-1}C_{i}  = SL_{2}(K)$.  The $C_{i}$ will be chosen sufficiently generically so that the theory $T^*$ of the structure $(K,+,\times, C_{1}, C_{2},C_{3}, C_{4})$ is simple of $SU$-rank $1$.
If by way of contradiction $G = SL_{2}(K)$ were definably amenable, witnessed by (left) invariant Keisler measure $\mu$, then the requirement above implies that $\mu(C_{i}) \geq 1/3$ for each $i\in [4]$ but then by disjointness, $\mu(G) \geq 4/3$ a contradiction.

In Section \ref{sec 4}, we will mention a closely related example with $F_{6}$ in place of $F_{12}$ but  with a partition of $SL_{2}(K)$ into six sets rather than four.  In terms of certain invariants related to ``definable paradoxical decompositions", this other example could be considered ``better". The general theory of paradoxical decompositions in both the abstract or discrete groups setting and the definable setting will also be discussed.

As in Section \ref{sec 2}, we will describe a universal theory $T$, and $T^{*}$ will be its model companion, but no longer complete. 

\subsection{The universal theory}\label{subs 3.1}
The language $L$ will be that of unital rings, together with four $4$-ary predicate symbols  $C_{1}, C_{2}, C_{3}, C_{4}$.  

It is well-known that

\[\begin{matrix} a = \left( \begin{matrix}
1 & 2 \\ 0 & 1 
\end{matrix} \right), &
b = \left( \begin{matrix}
1 & 0 \\ 2 & 1 
\end{matrix} \right)
\end{matrix}
\]
generate a free group in $\mathrm{SL}_{2}(\mathbb{Z})$.  Hence so do the matrices
\[ 
a^{-k} b a^{k}  =  \left(\begin{matrix} 1-4k & -8k^2 \\ 2 & 4k+1 \end{matrix} \right),
\]
for $k = 0, \ldots, 11$.  We number these 12 matrices in some way as $a(i,j)$, for $i \in [4]$, $j \in [3]$.  We will refer to the group generated by these matrices as $F_{12}$. Note that the entries of each $a(i,j)$ are terms of the language. 

For an integral domain $R$ of characteristic $0$, $SL_{2}(R)$ is the collection of $2\times 2$ matrices over $R$ of determinant $1$. 
The (universal)  theory $T$  in the language $L$ will be the theory of integral domains $R$ of characteristic $0$ together with axioms:
\begin{enumerate}[(i)]
\item The $4$-ary predicates $C_{1},\ldots,C_{4}$ partition $SL_{2}(R)$, and
\item For each $x\in SL_{2}(R)$ and each $i\in [4]$, there is $j\in [3]$ such that $a(i,j)\cdot x \in C_{i}$. 
\end{enumerate}

\subsection{Combinatorics and colourings}\label{subs 3.2}
We prove some lemmas needed for defining $T^{*}$.
The context in this section is simply the free group $G = F_{12}$ on $12$ generators  numbered as $a(i,j)$ for $i\in [4]$ and $j\in [3]$ together with a free action of $G$ on a set $X$.

\begin{Definition} Let $X_{0}\subseteq X$. A colouring $c: X_{0} \to [4]$ is {\em good} if for all $x\in X_{0}$ and $i\in [4]$, IF $a(i,j)\cdot x \in X_{0}$ for all $j\in [3]$, THEN
$c(a(i,j)\cdot x) = i$ for some $j\in [3]$. 
We call this condition the $ith$ colouring axiom at $x$.
Also we may call the (good) colouring total if $X_{0} = X$. 
\end{Definition}

As at the beginning of Section \ref{subs 2.2} we have a graph structure on $X$, relative to our choice above of free generators of $F_{12}$, the notion of a path between two points of $X$, and the distance function $d$. Recall that there will be at most one path between distinct points $x,y\in X$.  For $X_{0}$ a subset of $X$, $B_{n}(X_{0})$ is the ball of radius $n$ around $X_{0}$, namely  the set of $x\in X$ such that there is $y\in X_{0}$ with $d(x,y)\leq n$. 
We will {\em not} have an explicit analogue of the tree structure from Definition \ref{def 2.1}. 

We now mention some additional conventions and facts that we will make use of for the new example.
First  a subset $X_{0}$ is said to be {\em connected} if for any $x,y\in X$, $d(x,y)< \infty$ and all points on the  path from $x$ to $y$ are in $X_{0}$.  Note that this notion depends only on the graph structure on $X$.  

Given a subset $X_{0}$ of $X$, a maximal connected subset of $X_{0}$ will be called a {\em connected component} of $X_{0}$, and $X_{0}$ will be a disjoint union of its connected components.  Note that a connected component of $X$ itself is the same thing as an $F_{12}$-orbit.

Finally, given disjoint connected subsets $C_{0}$, $C_{1}$ of $X$, by a {\em path between $C_{0}$ and $C_{1}$} we mean a path  $x_{0}, x_{1},...,x_{n}$ between some $x_{0}\in C_{0}$ and some $x_{n}\in C_{1}$ such that no $x_{i}$ for $i=1,..,n-1$ is in $C_{0}\cup C_{1}$.  It is easy to check that if there is such a path, then it has to be unique, in which case we let $d(C_{0}, C_{1})$ be the length of such a path. 

We will use freely these notations and facts in the rest of Section \ref{subs 3.2}.



\vspace{2mm}
\noindent
We now give some lemmas about extending good colourings.

\begin{Lemma}\label{lem extending color 1 in sec 3} Suppose that $X_{0}\subseteq X$ is connected. Then any good colouring $c_{0} : X_{0} \to [4]$ extends to a total good colouring. 
\end{Lemma}
\begin{proof} We may assume that $X_{0} \neq \emptyset$, otherwise replace it by a singleton coloured with any colour.  As good colourings can be defined independently on connected components of $X$, we may  assume that $X$ is connected, so equals $\bigcup_{n}B_{n}(X_{0})$. And note that each $B_{n}(X_{0})$ is connected. We extend $c_{0}$ to $X$ by induction. Assume that we already have a good colouring $c_{n}: B_{n}(X_{0}) \to [4]$ extending $c_{0}$.  We extend to $c_{n+1}$. Suppose first that $y = a(i,j)\cdot x \in B_{n+1}(X_{0})\setminus B_{n}(X_{0})$ for some $x\in B_{n}(X_{0})$, and some $i,j$, then define $c_{n+1}(y) = i$. Note that this is well-defined, as $\{x,y\}$ is the unique path between the connected sets $B_{n}(X_{0})$ and $\{y\}$. 
If $y\in B_{n+1}(X_{0})\setminus B_{n}(X_{0})$ is not of the form, $a(i,j)x$ for $x\in B_{n}(X_{0})$,  define $c_{n+1}(y) \in [4]$ arbitrarily.   

We have to check that $c_{n+1}$ is a good colouring of $B_{n+1}(X_{0})$. Suppose $x\in B_{n+1}(X_{0})$ and  $i\in [4]$, and $a(i,j)\cdot x\in B_{n+1}(X_{0})$ for all $j\in [3]$.  Now if $a(i,j)\cdot x \in B_{n}(X_{0})$ for all $j\in [3]$ then by connectedness of $B_{n}(X_{0})$ also $x\in B_{n}(X_{0})$ and so as $c_{n}$ is a good colouring and $c_{n+1}$ extends $c_{n}$, the $ith$ colouring axiom  at $x$ is satisfied.   Otherwise  $a(i,j)\cdot x \in B_{n+1}(X_{0})\setminus B_{n}(X_{0})$ for some $j\in [3]$. Let $y = a(i,j)x$. Then $x\in B_{n}(X_{0})$, for if not, both $x$ and $y$ have distance $n+1$ from $X_{0}$, contradicting the existence of a unique path between the connected sets $X_{0}$ and $\{x,y\}$. Hence $c_{n+1}(y) = i$ by definition, and we have shown that the $c_{n+1}$ satisfies the $ith$ colouring axiom  at $x$.  As $x\in B_{n+1}(X_{0})$ and $i\in [4]$ were arbitrary we see that $c_{n+1}$ is a good colouring of $B_{n+1}(X_{0})$.
\end{proof}

\begin{Lemma}\label{lem extending color 2 in sec 3} Let $C_{0}$, $C_{1}$ be disjoint connected subsets of $X$ with $3 \leq d(C_{0}, C_{1})$ $< \infty$. Let $C$ be the smallest connected subset of $X$ containing $C_{0}\cup C_{1}$. Then any good colouring $c_{0}$ of $C_{0}\cup C_{1}$ extends to a good colouring of $C$.
\end{Lemma}
\begin{proof} Note that $C$ is the union of $C_{0}$, $C_{1}$ and the points on the unique path $I$ connecting them. By assumption the length of $I$ is $\geq 3$, namely  $|I|\geq 4$.  Now extending, if necessary,  $C_{0}$ to a suitable $B_{n}(C_{0})$ and extending $c_{0}|C_{0}$ to a good colouring of $B_{n}(C_{0})$ we may assume that $I = (u,v,y,z)$ with $u\in C_{0}$, $z\in C_{1}$ and $v,y\notin C_{0}\cup C_{1}$. 

If $v = a(i,j)\cdot u$ for some $i, j$ put $c(v) = i$.  Otherwise define it arbitrarily. Likewise if $y = a(i,j)\cdot z$ for some $i,j$ define $c(y) = i$. Note that this is well-defined.
We  have to check that $c$ is a good colouring.  And for this it is clear that we only need to check the $ith$ colouring axioms at $u,v,y, z$ (for all $i$). For $u, z$ it is clear by construction.  And for $v,y$ it is also clear vacuously, because it cannot be the case that all of $a(i,1)\cdot v$, $a(i,2)\cdot v$ and $a(i,3)\cdot v$ lie in $C$, and similarly for $y$. 
\end{proof}

\begin{Lemma}\label{lem extending color 3 in sec 3} Suppose $X_{0}\subseteq X$ has $n$ connected components, any two of which are of distance $\geq 2^{n}$ apart. Then any good colouring $c_{0}$ of $X_{0}$ extends to a good colouring of $X$.
\end{Lemma}
\begin{proof}  By induction on $n$. The case $n=1$ is Lemma \ref{lem extending color 1 in sec 3}.  The case $n=2$ is Lemma \ref{lem extending color 2 in sec 3}, noting that  $2^{2} = 4 \geq 3$. 

So let us assume $n\geq 2$ and the lemma holds for $n$ and we want to prove it for $n+1$. 
Let $X_{0}$ have  $n+1$ connected components  $C_{0}, \ldots, C_{n}$ and let $c_{0}$ be a good colouring of $X_{0}$. 
As the connected components of $X$ can be coloured separately, we may assume that the $C_{i}$ lie on a common connected component of $X$. We may also assume that the distance $l$ between $C_{0}$ and $C_{1}$ is the minimal distance between distinct pairs $C_{i}$, $C_{j}$. 
Let $C_{1}'$ be the smallest connected subset of $X$ containing $C_{0}$ and $C_{1}$  (as mentioned earlier $C_{1}'$ is the union of $C_{0}$, $C_{1}$ and the points on the unique shortest path between $C_{0}$ and $C_{1}$).  Using Lemma \ref{lem extending color 2 in sec 3}, let $c_{1}'$ be a good colouring of $C_{1}'$ extending  $c_{0}|(C_{0}\cup C_{1})$. 

\begin{Claim*} 
For each $i >1$, the distance between $C_{1}'$ and $C_{i}$ is at least $2^{n}$. 
\end{Claim*}

\noindent{\em Proof of Claim.}  Fix $i >1$ and let $d = d(C_{1}', C_{i})$ and suppose for a contradiction that $d< 2^{n}$. As $d(C_{0},C_{i})$ and $d(C_{1},C_{i})$ are both $\geq 2^{n+1}$, then $d$ has to be witnessed by $d(x,C_{i})$, where $x$ is a point on the unique shortest path $I$ between $C_{0}$ and $C_{1}$ which we know has length $l$.  So $d(x,C_{i})< 2^{n}$,  $d(x,C_{0}) = l_{0}$ say, and $d(x,C_{1}) = l_{1}$  say with $l_{0} + l_{1} = l$.   Moreover $d(C_{0}, C_{i}) \leq l_{0} + d$ and $d(C_{1}, C_{i}) \leq  l_{1} + d$, both of which are $\geq l$ by choice of $C_{0}$ and $C_{1}$. 
But then  $l + 2d = l_{0} + d + l_{1} + d  \geq 2l$ which implies $2d \geq l \geq 2^{n+1}$, which implies $d\geq 2^{n}$, a contradiction. 

Let $X_{0}' = C_{1}' \cup C_{2}\cup \ldots \cup C_{n}$, and let $c_{0}'$ be $c_{0}$ on $C_{2}\cup \ldots\cup C_{n}$ and $c_{1}'$ on $C_{1}'$. Note that $c_{0}'$ is a good colouring on $X_{0}'$ as it is good on each connected component of $X_{0}'$. Then by the claim, and  the induction hypothesis, $c_{0}'$ extends to a good colouring $c$  of $X$, and as $c_{0}'$ extends $c_{0}$, $c$ extends $c_{0}$ too. 
\end{proof}

\begin{Lemma}\label{lem extending color 4 in sec 3} Suppose $X_{0}\subseteq X$  has size $n$. Let $\alpha(n) = 2^{n+1} - 1$, and let $c_{0}: X_{0}\to [4]$ be a good colouring which extends to a good colouring $c': B_{\alpha(n)}(X_{0}) \to [4]$. Then $c_{0}$ extends to a good colouring $c: X \to [4]$ of $X$. 
\end{Lemma}
\begin{proof} Let $k_{0}$ be the number of connected components of $X_{0}$. So $k_{0} \leq n$.

\noindent{\em Case 1.}  Either $k_{0} = 1$ ($X_{0}$ is connected) or $k_{0} > 1$ and the $k_{0}$ connected components of $X_{0}$ are at distance $\geq 2^{k_{0}}$ apart.

\noindent 
Then by Lemma \ref{lem extending color 3 in sec 3}, $c_{0}$ extends to a good colouring of $X$. And we are finished.

\noindent {\em Case 2.}  Otherwise. Then define  $X_{1} = B_{2^{k_{0}}}(X_{0})$, and $k_{1}$ to be the number of connected components of $X_{1}$.  And note that $k_{1} < k_{0}$ and $X_{1}\subseteq B_{\alpha (n)}(X_{0})$.

\noindent
Again if either $X_{1}$ is connected or the $k_{1}$ connected components of $X_{1}$ are of distance $\geq 2^{k_{1}}$ apart,  then the good colouring $c'|X_{1}$ extends to a good colouring of $X$, and we finish. 

Otherwise define $X_{2} = B_{2^{k_{1}}}(X_{1})$ and $k_{2}$ to be the number of connected components of $X_{2}$. So $k_{2} < k_{1}$.

We continue this way to produce  $k_{0} > k_{1} > \ldots > k_{l} \geq 1$ and $X_{0} \subseteq X_{1} \subseteq \ldots \subseteq X_{l}$ where $X_{i}$ has $k_{i}$ connected components, until 
we get that $X_{l}$ is connected or its $k_{l}$ connected components are at distance $\geq 2^{k_{l}}$ apart, and we extend $c'|X_{l}$ to a good colouring of $X$.

We have to check why the process can be continued, in particular why  each $X_{i}\subseteq  B_{\alpha(n)}(X_{0})$.  It is because, $k_{i} \leq n- i$ for each $i$, and so 
$\sum_{i=0,\ldots,l}2^{k_{i}} \leq \sum_{i=0,\ldots,n} 2^{n-i} = \sum_{i=0,\ldots,n} 2^{i} = 2^{n+1} - 1 = \alpha(n)$. Whereby  $X_{i}\subseteq B_{\alpha(n)}(X_{0})$ for all $i=1,\ldots,l$. 
\end{proof}

\subsection{The theory $T^{*}$}\label{subs 3.3}
Here we will obtain the model companion $T^{*}$ of the universal theory $T$ introduced in Section \ref{subs 3.1}.
In terms of compatibility with notation in the previous section, we will write a model of $T$ as $M = (R,c)$, where $R$ is an integral domain of characteristic $0$ and $c$ is the  colouring $SL_{2}(R)\to [4]$ such that $C_{i}(M) = c^{-1}(i)$ for $i=1,\ldots,4$. So as $F_{12}$ is acting freely on $SL_{2}(R)$ by left multiplication, the axioms from Section \ref{subs 3.1} say precisely that $c$ is a good colouring.  In this context we will use freely the colouring notation from the previous section. 
We begin with some observations which will be useful for the rest of Section \ref{sec 3}.
\begin{Lemma}\phantomsection\label{lem extending color to integral domain} 
\begin{enumerate}[(i)]
\item 
Let $R$ be an integral domain, let $X\subseteq SL_{2}(R)$ be a union of $F_{12}$-orbits (connected components of $SL_{2}(R)$), and let $c$ be a colouring of $X$ whose restriction to each $F_{12}$-orbit is good. Then $c$ extends to a good colouring of $SL_{2}(R)$.
\item
Let $R\subseteq S$ be integral domains. Then $S_{2}(R)$ is a union of $F_{12}$-orbits, and any good colouring of $SL_{2}(R)$ extends to a good colouring of $SL_{2}(S)$.
\item
Let $K\subseteq K_{1}, K_{2}$ be algebraically closed fields, such that $K_{1}$ is independent from $K_{2}$ over $K$ (in some ambient algebraically closed field, and in the sense of $ACF_{0}$). Let $L$ be the compositum of (the field generated by) $K_{1}$ and $K_{2}$.  Then $SL_{2}(L)$ is a disjoint union of $F_{12}$-orbits contained in $SL_{2}(K)$, $F_{12}$-orbits contained in $SL_{2}(K_{1})\setminus SL_{2}(K)$, $F_{12}$-orbits contained in $SL_{2}(K_{2})\setminus SL_{2}(K)$, and $F_{12}$-orbits contained in $SL_{2}(L)\setminus (SL_{2}(K_{1})\cup SL_{2}(K_{2}))$. 
\end{enumerate}
\end{Lemma}
\begin{proof} (i)  First note that by the definition of a good colouring $c$ is a good colouring of $X$.  Each $F_{12}$ orbit in $SL_{2}(R)$ which does not intersect $X$ is disjoint from $X$ and has a good colouring, by Lemma \ref{lem extending color 1 in sec 3}.  And again all these good colourings, together with $c$ give a good colouring of $SL_{2}(R)$.
\newline
(ii)  is immediate, using (i).
\newline
(iii)  The independence assumption tells us that $SL_{2}(K_{1})\cap SL_{2}(K_{2}) = SL_{2}(K)$, from which everything else follows. 
\end{proof}

By part (ii) of the above lemma, if  $(R,c)$ is an existentially closed model of $T$, then $R$ is an algebraically closed field.  From now on we will assume that $R = K$ is an algebraically closed field, and we situate $K$ in a larger saturated algebraically closed field $\tilde K$  from which we can choose generic points of algebraic varieties over $K$ (and write $SL_2$ for $SL_2(\tilde K)$). 

For technical reasons related to a subsequent relative quantifier elimination proof by a back and forth argument we will be concerned with extending the colouring $c$ of $SL_{2}(K)$ to generic points of curves on $(SL_{2})^{n}$. 
Here by a curve on $(SL_{2})^{n}$ over $K$, we mean an (absolutely) irreducible curve $C\subseteq (SL_{2})^{n}$, defined over $K$, for some $n$.  We will call $C$ a {\em good curve} over $K$, or {\em good $K$-curve}, if in addition
if $d = (d_{1},\ldots,d_{n})$ is a generic point of $C$ over $K$, then each  $d_{i}\notin SL_{2}(K)$. 

In the following $\alpha(n) = 2^{n+1} - 1$ as in Lemma \ref{lem extending color 4 in sec 3}.

\begin{Definition} Let $K$ be an algebraically closed field.  Let $n\geq 1$, let $C\subseteq (SL_{2})^{n}$ be a good $K$-curve, and let $c_{0}:[n]\to [4]$. We will say that $C$ is {\em safe for $c_{0}$ over $K$} if  for $d= (d_{1},\ldots,d_{n})$ a generic point of $C$ over $K$, the colouring ${\tilde c}:  \{d_{1},\ldots,d_{n}\} \to [4]$ defined by ${\tilde c}(d_{i}) = c_{0}(i)$, extends to a good colouring ${\tilde c}'$ of  $B_{\alpha(n)}(\{d_{1},\ldots,d_{n}\})\subseteq SL_{2}(\tilde K)$. 
\end{Definition}

Fix $n$. Let us now fix a (quantifier-free) formula $\phi({\bar x}, {\bar y})$ in the language $L_{r}$  of rings such that for any algebraically closed field $F$ and tuple ${\bar a}$ from $F$ (whose length is the same as the length of $\bar y$), $\phi({\bar x}, {\bar a})$, if consistent, defines a good $F$-curve  $D_{\bar a}\subseteq  SL_{2}(F)^{n}$.   We call such $\phi({\bar x},{\bar y})$ a ``\emph{good formula}". 

\begin{Remark}\label{rem good formula} Note that for any algebraically closed field $F$ and good curve  $D\subseteq SL_{2}(F)^{n}$, there is a good formula $\phi({\bar x}, {\bar y})$ and ${\bar a}\in F$ such that $D = D_{\bar a}$.  This is because we can express dimension and irreducibility of algebraic varieties, and we can also express that the projection of a curve onto each coordinate has infinite image. 
\end{Remark}

\begin{Lemma}\label{lem curve is safe is definable} Given $n$, good formula $\phi({\bar x}, {\bar y})$ as above, and a function $c_{0}:[n]\to [4]$, there is a formula $\psi({\bar y})$ in $L_{r}$, such that for every algebraically closed field $K$  and ${\bar a}\in K$, $K\models \psi({\bar a})$ iff  the curve $D_{\bar a}$ is safe for $c_{0}$ over $K$. 
\end{Lemma}
\begin{proof} Note that we are working completely in the language of rings, even though we mention colourings. 
 First note that for a curve $C\subseteq  SL_{2}(K)^{n}$ and any $(d_{1},\ldots,d_{n})\in C(K)$, there is a bound $\kappa_{n}$ on the cardinality of $B_{\alpha(n)}(\{d_{1},\ldots,d_{n}\})$, and
moreover by a case analysis we can identify definably, from properties of the $d_{i}$ the precise cardinality.
There is a formula  $\chi(z_{1},\ldots,z_{n})$ in $L_{r}$ expressing that  $c$ is a good coloring of 
$B_{\alpha(n)}(\{z_{1},\ldots,z_{n}\})$ into $4$ colours $\{1,2,3,4\}$ such that $c(z_{i}) = c_{0}(i)$.

We now bring in the good formula $\phi({\bar x}, {\bar y})$.  Let $\psi({\bar y})$ express that for infinitely many ${\bar x}$ such that  $\phi({\bar x}, {\bar y})$ holds, $\chi({\bar x})$ 
holds. 
 Then  for $K$ algebraically closed, and ${\bar a}\in K$, $K\models \psi({\bar a})$ iff for generic ${\bar d}$ on $D_{\bar a}$ over $K$,  there is a good colouring $c$ of $B_{\alpha(n)}(\{d_{1},\ldots,d_{n}\})$ such that $c(d_{i}) = c_{0}(i)$ for $i=1,\ldots,n$, namely that $D_{\bar a}$ is safe for $c_{0}$ over $K$. 
\end{proof} 

We can now define $T^{*}$.
\begin{Definition}  $T^{*}$ is the $L$-theory expressing of  $(K,c)$, that:
\begin{enumerate}[(i)]
\item $K$ is algebraically closed and $(K,c)\models T$;
\item whenever $C\subseteq (SL_{2})^{n}$ is a good curve over $K$, $c_{0}:[n]\to [4]$ and $C$ is safe for $c_{0}$, then there are infinitely many $d = (d_{1},\ldots,d_{n})\in C(K)$ such that $c_{0}(i) = c(d_{i})$ for $i=1,\ldots,n$.
\end{enumerate} 
\end{Definition}

\begin{Remark}  By Remark \ref{rem good formula} and Lemma \ref{lem curve is safe is definable}, the property (ii) in the definition of $T^*$ above is expressed by an axiom schema, ranging over $n$ and good formulas $\phi({\bar x}, {\bar y}) \in L_r$. 
\end{Remark}

\begin{Lemma}\label{lem extending to a model of star} Any model  $(R,c)$ of $T$ extends to a model $(F, c')$ of $T^{*}$. In particular $(T^{*})_{\forall} = T$ and $T^{*}$  is consistent. 
\end{Lemma}
\begin{proof}  Fix $(R,c)\models T$ and as mentioned after Lemma \ref{lem extending color to integral domain} we may assume $R = K$ to be an algebraically closed field. We will fix a good curve $C\subset (SL_{2})^{n}$ over $K$ and $c_{0}:[n] \to [4]$, such that $C$ is safe for $c_{0}$,  and find an extension $(F,c')$ of $(K,c)$ and $d = (d_{1},\ldots,d_{n})\in C(F)$ such that $c'(d_{i}) = c_{0}(i)$ for $i=1,\ldots,n$. We will also choose $F$ algebraically closed.  So in $(F,c')$ we satisfy Axiom Schema (i) as well as a weaker form of one instance of  the Axiom Schema (ii) for $T^{*}$, namely that there is at least one, rather than infinitely many, $d$ satisfying the required conditions.  Extending $(K,c)$  to a model of $T^{*}$ is then a routine union of chains argument, including finding the infinitely many $d$ as above.  Details are left to the reader. 

Simply choose $d = (d_{1},\ldots,d_{n})$ to be a point of $C$ in $\tilde K$ generic over $K$.  By goodness of $C$, each $d_{i}\in SL_{2}(\tilde K)\setminus SL_{2}(K)$.  By assumption there is a good 
colouring $c''$ of  $B_{\alpha(n)}(\{d_{1},\ldots,d_{n}\})$ such that  $c''(d_{i}) = c_{0}(i)$ for $i=1,\ldots,n$.
Let $F$ be the algebraic closure of  the field generated by $K$ and $d$.  And let $X =  SL_{2}(F)\setminus SL_{2}(K)$.  Then $X$ is a union of $F_{12}$-orbits and  $B_{\alpha(n)}(\{d_{1},\ldots,d_{n}\}) \subset 
X$. Hence, by Lemma \ref{lem extending color 4 in sec 3}, there is a good colouring $c'''$ of $X$ with $c'''(d_{i}) = c_{0}(i)$ for $i=1,\ldots,n$.  As $X$ and $SL_{2}(K)$ are both unions of $F_{12}$-orbits, $c\cup c'''$ will 
be a good colouring of $SL_{2}(F)$ extending $c$. Denote $c\cup c'''$ by $c'$,  and we have produced our required extension $(F,c')$ of $(K,c)$.
\end{proof}

\begin{Lemma}\label{lem: back n forth Sec 3}  Let $(F_{1},c_{1})$, $(F_{2}, c_{2})$ be $\aleph_{1}$-saturated models of $T^{*}$. Let $I$ be the collection of partial isomorphisms between (nonempty) countable substructures  of $F_{1}$, $F_{2}$ respectively which are of the form $(K_{1}, c_{1}|K_{1})$, $(K_{2}, c_{2}|K_{2})$  where $K_{1}, K_{2}$ are algebraically closed fields.
Then $I$ has the back-and-forth property. 
\end{Lemma}
\begin{proof} Suppose that we are given an isomorphism $f$ between $(K_{1},c_{1}|K_{1})$ and  $(K_{2}, c_{2}|K_{2})$. It is enough to extend $f$ to  $g$ with domain $L_{1}\supseteq K_{1}$ where $L_{1}$ is algebraically closed and of transcendence degree  $1$ over $K_{1}$.  By compactness, it suffices to prove the following.

\begin{Claim*}
For every finite tuple $d_{1},\ldots,d_{n}$ from $SL_{2}(L_{1})$ there are $e_{1},\ldots,e_{n}$ in $SL_{2}(F_{2})$ such that  the map $g$ which extends $f$ and takes $d_{i}$ to $e_{i}$ for $i=1,\ldots,n$  preserves quantifier-free $L_{r}$-types, as well as satisfying $c_{2}(e_{i}) = c_{1}(d_{i})$ for $i=1, \ldots, n$. 
\end{Claim*}
\noindent{\em Proof of Claim.}  We may clearly assume that  $d_{1},\ldots,d_{n}\notin SL_{2}(K_{1})$ for $i=1,\ldots,n$. 
 It follows that $(d_{1},\ldots,d_{n})$ is a generic over $K_{1}$ point of a good curve $C_{1}\subset SL_{2}^{n}$ over $K_{1}$. 
Let $c_{0}:[n]\to [4]$ be defined by $c_{0}(i) = c_{1}(d_{i})$.  Hence $C_{1}$ is safe for $c_{0}$ over $K_{1}$.  As $f$ is an isomorphism of algebraically closed fields, the curve $C_{2} 
= f(C_{1})$ is safe for $c_{0}$ over $K_{2}$. In particular $C_{2}$ is safe for $c_{0}$ over $F_{2}$.  However $(F_{2}, c_{2})$ is a model of $T^{*}$, so Axiom Schema (ii) implies that there are 
infinitely many $e = (e_{1},\ldots,e_{n})\in C_{2}(F_{2})$ such that $c_{0}(i) = c_{2}(e_{i})$ for $i=1,\ldots,n$.  By $\aleph_{1}$-saturation of $(F_{2},c_{2})$  (and countability of $K_{2}$) we 
can find  $e = (e_{1},\ldots,e_{n})\in SL_{2}(F_{2})$ a generic over $K_{2}$ point of $C_{2}$ such that  $c_{0}(i) = c_{2}(e_{i})$ for $ i=1,\ldots,n$. 
As the  quantifier-free $L_{r}$-type of  $e$ over $K_{2}$ is the image under $f$ of the quantifier-free $L_{r}$-type of $d$ over $K_{1}$, and $c_{1}(d_{i}) = c_{0}(i) = c_{2}(e_{i})$ for $i=1, \ldots ,n$ we have proved the claim, and hence the lemma. 
\end{proof} 

\begin{Theorem}\phantomsection\label{cor of back and forth sec 3} 
\begin{enumerate}[(i)]
\item Let ${\bar a} = (a_{\alpha}: \alpha < \gamma)$, ${\bar b} = (b_{\alpha} : \alpha < \gamma)$ be tuples of the same length $\gamma$ in models $M, N$ of  $T^{*}$, where $\gamma$ is an ordinal. Then $tp_{M}({\bar a}) = tp_{N}({\bar b})$ iff the map taking $a_{\alpha}$ to $b_{\alpha}$  for $\alpha < \gamma$ extends to an isomorphism between the substructures  $(K,c)$ of $M$ and $(K',c')$ of $N$ where $K = acl({\bar a})$ and $K' = acl({\bar b})$ in the sense of fields.
\item  In a model $M$ of $T^{*}$, the model theoretic algebraic closure of a subset $A$ of $M$ coincides with the (field theoretic) algebraic closure of the field generated by $A$.
\item The completions of $T^{*}$ are determined by the isomorphism types of the algebraic closure of $\Q$ equipped with an $L$-structure.
\item $T^{*}$ is the model companion of $T$.
\end{enumerate} 
\end{Theorem}
\begin{proof} (i) is an immediate consequence of Lemma \ref{lem: back n forth Sec 3}.  

(ii). In the light of (i) we have to check that if $M$  is a saturated model of $T^{*}$ and $(K,c)$ is a (small) substructure of $M$ where $K$ is algebraically closed as a field, then for any $a\in M\setminus K$,  there are infinitely many realizations of the type of $a$ over $K$ in the sense of the ambient model $M$ of $T^{*}$.   Let $K'$ be the (field-theoretic)  algebraic closure in $M$ of the field $K(a)$. Then $(K',c|K')$ is an $L$-structure whose isomorphism type determines its type by (i).
Now we build abstractly another ``algebraically closed" model of $T$, as follows.  Let  $\tilde K$ be a large algebraically closed field containing $K$ and let  $(a_{i}:i<\omega)$ in $\tilde K$ 
be algebraically independent over $K$.  Let $K'_{i}$ be the (field-theoretic) algebraic closure of $K(a_{i})$.  Fix field isomorphisms  $f_{i}$ of $K'$ with $K'_{i}$ over $K$ which take $a$ to 
$a_{i}$, and use these to copy the additional structure (the colouring) to the $K'_{i}$. So each $K'_{i}$ is equipped with a good colouring $c_{i}$ extending $c$ on $K$.  Let $F$ be the 
field generated by $\bigcup_i K'_{i}$. Notice that  $\bigcup_{i}SL_{2}(K'_{i})$ is a union of $F_{12}$-orbits inside $F$ and $\bigcup_{i}c_{i}$ gives a good colouring of this union.  Hence by Lemma \ref{lem extending color to integral domain}, we can extend $\bigcup_{i} c_{i}$ to a good colouring $c'$ of $SL_2(F)$ to get $(F,c')\models T$.  Embed $(F,c')$  in a model $N$ of $T^{*}$, and we see by (i), 
that each $a_{i}$ has the same type over $K$ in $N$, which also equals  $tp(a/K)$ in $M$. 

(iii) is a special case of (i) (for the empty tuples).

(iv)  is another special case of (i): let $M\subseteq N$ be models of $T^{*}$. Then the identity map $M\to N$ is an isomorphism of $L$-structures whose underlying field is algebraically closed, hence an elementary map by (i). So $T^{*}$ is model complete, hence by Lemma \ref{lem extending to a model of star} is  the model companion of $T$.
\end{proof}

\subsection{Simplicity and the proof of Theorem \ref{thm Main theorem 2}.}\label{subs 3.4}

To prove simplicity of the theory $T^{*}$ (namely of any of its completions) we will make use of Theorem 4.2 from \cite{Kim-Pillay}  which says that  it suffices to prove that one has a ``notion of independence" which satisfies the Independence Theorem over a model.  See \cite{Kim-Pillay} for the notion of independence.  The Independence Theorem over a model states that (in the context of a saturated model $\bar M$ of a complete theory), if $M$ is a small elementary substructure of $\bar M$, and $a,b, d_{0}, d_{1}$ are tuples such that $a$ and $b$ are independent over $M$, $d_{0}$ and $a$ are  independent over $M$, $d_{1}$ and $b$ are independent over $M$, and $tp(d_{0}/M) = tp(d_{1}/M)$, THEN  there is $d$ realizing $tp(d_{0}/M,a)$ as well as $tp(d_{1}/M,b)$ such that $d$ is independent from $M, a, b$ over $M$. 

\begin{Proposition}\label{prop completions of star are simple sec 3} Every completion of $T^{*}$ is simple, and nonforking independence coincides with independence in the sense of the reduct to $ACF_{0}$. In particular the $SU$-rank of $x=x$ is $1$. 
\end{Proposition}
\begin{proof} Fix a saturated model $\bar M$ of $T^{*}$.  We let $c$ denote the colouring on $\bar M$. 

Types will refer to types in $\bar M$ and $tp_{ACF}$ to types in the reduct of ${\bar M}$ to the field language.  We will  prove that  $ACF$-independence is a notion of independence which satisfies the Independence Theorem over a model, as described above.   The only nontrivial thing to check in terms of being a notion of independence is the extension property, but it follows easily from Theorem \ref{cor of back and forth sec 3}(ii), or by our method of proof below of the Independence Theorem. So it remains to prove that
$ACF$-independence  in $\bar M$ satisfies the Independence Theorem over a model: namely suppose $M$ is a small elementary substructure of $\bar M$ and $a,b, d_{0}, d_{1}$ are tuples such that $a$ and $b$ are $ACF$ independent over $M$, $d_{0}$ and $a$ are $ACF$ independent over $M$, $d_{1}$ and $b$ are $ACF$-independent over $M$, and $tp(d_{0}/M) = tp(d_{1}/M)$, THEN  there is $d$ realizing $tp(d_{0}/M,a) \cup tp(d_{1}/M,b)$ such that $d$ is $ACF$-independent from $M, a, b$ over $M$. 
Let $D_{0} = acl(d_{0}M)$, $D_{1} = acl(d_{1}M)$, $A = acl(aM)$ and $B = acl(bM)$. In spite of the notation we will enumerate $D_{0}$, $D_{1}$, $A$ and $B$  (and other sets introduced below) in a consistent fashion (vis-a-vis, $d_{0}, d_{1}, a, b)$  as tuples and treat them as such. In particular $D_{0}$ and $D_{1}$ will have the same type over $M$ in the structure $\bar M$ so also in the $ACF$ reduct.  By stationarity of this type in the $ACF$-reduct, if $D$ realizes this $ACF$-type,  $ACF$-independently from $A\cup B$ over $M$ then $D$ realizes $tp_{ACF}(D_{0}/A)$ as well as $tp_{ACF}(D_{1}/B)$. 

Let $\sigma_{0}$ be a (field) isomorphism between $acl(D_{0}A)$ and $acl(DA)$ over $A$  (again treating these consistently as tuples), and likewise $\sigma_{1}$ an isomorphism between $acl(D_{1}B)$ and $acl(DB)$. Use the isomorphisms $\sigma_{0}$ and $\sigma_{1}$ to transport the colourings of $SL_{2}(acl(D_{0}A))$ and $SL_{2}(acl(D_{1}B))$ (coming from the structure $\bar M$) to
$SL_{2}(acl(DA))$ and $SL_{2}(acl(DB))$, which we call $c_{0}$ and $c_{1}$. 
 Let $F$ be the subfield of $\bar M$ generated by $acl(AB)$, $acl(DA)$ and $acl(DB)$.  Note that $c_{0}$ and $c_{1}$ agree on $D$, $c_{0}$ and $c$ agree on $A$, and $c_{1}$ and $c$ agree on $B$.   As we have that $D$ is $ACF$-independent from $AB$ over each of $A$, $B$, and $A$ is $ACF$-independent from $B$ over $D$, we conclude using Lemma \ref{lem extending color to integral domain} that the colouring $c'$ obtained by taking the union of $c|acl(AB)$, $c_{0}$ and $c_{1}$, is well-defined, and extends to a good colouring $c''$ of $SL_{2}(F)$.
By Lemma \ref{lem: back n forth Sec 3} we can embed $(F,c'')$ into $\bar M$ over $acl(AB)$, by a map $\sigma$. Let $D' = \sigma(D)$. So $D'$ is $ACF$-independent from $AB$ over $M$.
\begin{Claim*} 
$D'$ realizes $tp(D_{0}/A) \cup tp(D_{1}/B)$.
\end{Claim*}
\noindent {\em Proof of Claim.} We let $alg(C)$ denote the field-theoretic algebraic closure of the subfield of $\bar M$ generated by $C$.

Then $\sigma\circ\sigma_{0}(alg(D_{0}A)) = alg(D'A)$, and for every $e\in SL_{2}(alg(D_{0}A))$ we have that $c(e) = c_{0}
(\sigma_{0}(e)) = c'(\sigma_{0}(e)) = c(\sigma\circ\sigma_{0}(e))$. Thus we have an isomorphism over $A$ between 
\begin{gather*}
	\left( alg(D_{0}A), c|SL_{2}(alg(D_{0}A)) \right) \textrm{ and } (alg(D'A), c|SL_{2}(alg(D'A))).
\end{gather*}
Hence by Theorem \ref{cor of back and forth sec 3}(i), $D'$ realizes $tp(D_{0}/A)$.  By a similar proof, $D'$ realizes $tp(D_{1}/B)$. 

This proves the claim as well as the proposition. 
\end{proof}

\begin{proof*}{Proof of Theorem \ref{thm Main theorem 2}: existence of a non definably amenable group definable in a simple theory}
This is precisely as mentioned in the introduction to Section 3:   Fix a model $M = (K,c)$ of $T^{*}$ and let $T^{**}$ be the complete theory of $M$.  Proposition \ref{prop completions of star are simple sec 3} says that $T^{**}$ is simple of $SU$-rank $1$.  Let $G = SL_{2}(K)$ as a group definable in $M$ and we use notation as in Section \ref{subs 3.1}.   Assume for the sake of contradiction that $\mu$ is a left invariant Keisler measure on $G$. Fix an arbitrary $a \in G$ and $i\in [4]$. Then  by Axiom (ii) (of $T$), 
$G =  a(i,1)^{-1}C_{i} \cup a(i,2)^{-1}C_{i} \cup a(i,3)^{-1}C_{i}$. So by invariance of $\mu$, $\mu(C_{i}) \geq 1/3$.  On the other hand the $C_{i}$ for $i\in [4]$ partition $G$, whereby  $\mu(G) \geq  4/3$, a contradiction. 
\end{proof*}

\begin{proof}[Proof of Corollary \ref{cor simple and non amen}.]
Let  $M = (K,c)$ be a saturated model of $T^{*}$. Adjoin an ``affine copy" of $SL_{2}(K)$ as a new sort. Namely add a new sort $S$ together with a regular action of $SL_{2}(K)$ on $S$, to get a (saturated) structure $M'$. Then there is a unique $1$-type over $\emptyset$ realized in $S$.   Any automorphism invariant Keisler measure on the sort $S$ would yield a translation invariant Keisler measure on $SL_{2}(K)$, contradicting Theorem \ref{thm Main theorem 2}.  
\end{proof}

\begin{Remark}
Combining the proof of Theorem \ref{thm Main theorem 2} with the setting of \cite{Chatzidakis-Pillay}, it should be possible to obtain the following generalization of Theorem \ref{thm Main theorem 2}. Let $T$ be a simple model complete theory eliminating $\exists^{\infty}$ quantifier, and $G$ a definable group containing a non-abelian free subgroup  (as an abstract group, not necessarily definable). Then there exists a simple theory $T^*$ expanding $T$ so that forking in $T^*$ coincides with forking in the reduct $T$ (in particular, $T^*$ has the same SU-rank as $T$) and $G$ is not definably amenable in $T^*$.
\end{Remark}

\section{Paradoxical decompositions and additional results.}\label{sec 4}

Lying behind the second example (and also in a sense the first example) is the theory of ``definable paradoxical decompositions" from \cite{NIPI}, giving  necessary and sufficient conditions for a group $G$ definable in a structure $M$ to be definably amenable. When the structure $M$ is a model of set theory and $G$ is just a group, or just when  {\em all} subsets of $G$ are definable, then we are in the context of amenability of a discrete group $G$, and where there are classical results giving equivalent conditions.
In any case the theory of definable paradoxical decompositions gives some interesting invariants of non definably amenable groups and we can ask about the invariants of the  example in Section \ref{sec 3}. This and various other things are discussed in this final section.

\subsection{Definable paradoxical decompositions}\label{subs 4.1}

Let us first recall the (classical) notion of a {\em paradoxical decomposition} of a discrete or abstract  group $G$. We will abbreviate this notion as $cpd$ for ``classical paradoxical decomposition". A $cpd$ for $G$ consists of pairwise disjoint subsets  $X_{1},\ldots,X_{m}, Y_{1},\ldots,Y_{n}$ of $G$ and $g_{1},\ldots,g_{m}, h_{1},\ldots,h_{n}\in G$ such that  $G$ is the union of the $g_{i}X_{i}$ and is also the union of the $h_{j}Y_{j}$.  Recall that the discrete group $G$ is said to be {\em amenable} if there is a (left) translation invariant finitely additive probability measure on the collection (Boolean algebra) of {\em all} subsets of $G$. The well-known theorem of Tarski is:
\begin{Fact} Let $G$ a be group. Then $G$ is amenable if and only if $G$ has no paradoxical decomposition. 
\end{Fact}

\begin{Remark}\label{rem par dec partition of G} Clearly after replacing the $X_{i}, Y_{j}$ by suitable subsets, we can assume that each of the  $(g_{i}X_{i})_{i}$ and $(h_{j}Y_{j})_{j}$ form {\em partitions} of $G$.
\end{Remark}

One could ask whether  for a definable group $G$ (essentially a group equipped with a certain Boolean algebra of subsets, closed under left translation), we have the identical result: $G$ is definably amenable iff $G$ has a {\em definable} $cpd$, namely where the $X_{i}, Y_{j}$ are definable?  We expect the answer is no. In any case Tarski's proof  of ``nonamenability implies the  existence of a $cpd$" is nonconstructive and does not go over immediately to a definable version. 

In \cite{NIPI} there is another version of paradoxical decomposition which does give a characterization of definable amenability, remaining in the Boolean algebra of definable sets.

We will briefly describe this here.  We fix a definable group $G$ in a structure $M$. Definable will mean {\em with parameters}. 

By a \emph{($m$-)cycle} (for $m\geq 0$) we mean a formal sum 
$\sum_{i=1,\ldots,m} X_{i}$ of definable subsets $X_{i}$ of $G$.
If all the $X_{i}$ are the same we could write this formal sum as $mX_{i}$.  We can add such  cycles in the obvious way to get the ``free abelian monoid" generated by the definable subsets of $G$. 
And any definable subset $X$ of $G$ (including $G$ itself) is of course a (1-)cycle. 

If $X = \sum_{i=1,\ldots,m} X_{i}$ and $Y = \sum_{j=1,\ldots,n}Y_{j}$ are two cycles, then by a {\em definable piecewise} 
translation $f$ from $X$ to $Y$ we mean a map $f$ from the formal disjoint union $X_{1}\sqcup \ldots \sqcup X_{m}$ to  the formal disjoint union $Y_{1}\sqcup \ldots \sqcup Y_{n}$ for which there is a partition of each $X_{i}$ into definable subsets  $X_{i1}, \ldots, X_{in_{i}}$, and for each $i$ and $t\leq n_{i}$, an element $g_{it}$ of $G$ such that the restriction $f|X_{it}$ of $f$ to $X_{it}$ is just left translation by $g_{it}$, and $g_{it}X_{it}$ is a subset of one of the $Y_{j}$'s.  By a {\em definable} map from $X$ to $Y$ we mean just the same thing except that translation by $g_{it}$ on $X_{it}$ is replaced by a definable function with domain $X_{it}$ and image contained in some $Y_{j}$. 

Such a definable piecewise translation  (or definable map) $f$ is said to be {\em injective} if it is injective as a map between formal disjoint unions. So for example, in the case of definable piecewise translations this would mean that for each $i, i' \leq m$ and $t \leq n_{i}, t'\leq n_{i'}$ if $f$ takes both $X_{it}$ and $X_{i't'}$ into the same $Y_{j}$, then for $x\in X_{it}$ and $x'\in X_{i't'}$, $f(x) = f(x')$  implies that $i = i'$, $t = t'$ and $x=x'$. 


We write $X\leq Y$ if there is an injective piecewise definable translation $f$ from $X$ to $Y$.  Note that $\leq$ is reflexive and transitive. Also  $X\leq W$ and $Y\leq Z$ implies $X+Y \leq W + Z$. 

\begin{Definition} By a {\em definable paradoxical decomposition} ($dpd$) of the definable group $G$ we mean an injective  definable piecewise translation from $G + Y$ to $Y$ for some cycle $Y$. 
\end{Definition}

The following is proved in \cite{NIPI} (Proposition 5.4). 

\begin{Fact}\label{fact def amen iff no dpd}  $G$ is definably amenable if and only if $G$ does not have a $dpd$.
\end{Fact}

\begin{Lemma}\label{lem par dec basics} Suppose $G + Y \leq Y$ where $Y = \sum_{i=1,\ldots,n}Y_{i}$ (with the $Y_{i}$ definable). Then:
\begin{enumerate}[(i)]
\item  $mG + Y \leq Y$ for all $m\geq 1$,
\item $2Y \leq Y$,
\item  $(n+1)G \leq nG$. 
\end{enumerate}
\end{Lemma}
\begin{proof}  
(i)  By induction: $G+ Y \leq Y$ implies $(m + 1) G + Y  = mG + G+Y \leq mG + Y \leq Y$. 

(ii)  $Y + Y \leq nG+Y \leq Y$ (by taking $n=m$ in part (i)). 

(iii)  $(n+1)G \leq (n+1)G +Y \leq Y $ (by (i)) $= \sum_{i=1,\ldots,n}Y_{i} \leq \sum_{i=1,\ldots,n}G = nG$. 
\end{proof}

\begin{Corollary}\label{cor dpd iff n+1<n} $G$ has a $dpd$ iff $(n+1)G\leq nG$ for some $n\geq 1$.
\end{Corollary}

On the other hand:
\begin{Lemma}\label{lem def cpd iff 2<1}  $G$ has a definable $cpd$ (a $cpd$ where the $X_{i}$ and $Y_{j}$ are definable) if and only if $2G\leq G$ (if and only if $(n+1)G \leq nG$ for all $n$). 
\end{Lemma} 
\begin{proof} Suppose $G = \bigcup_{i=1,\ldots,m}g_{i}X_{i} = \bigcup_{j=1,\ldots,n}h_{i}Y_{j}$ witnesses a definable $cpd$. As mentioned in Remark \ref{rem par dec partition of G},  by replacing the $X_{i}$ and $Y_{j}$ by suitable subsets we can assume pairwise disjointness of the $g_{i}X_{i}$, as well as pairwise disjointness of the $h_{j}Y_{j}$, and we get  that $G + G \leq \bigcup_{i}X_{i} \cup \bigcup_{j}Y_{j} \leq G$.

The converse works the same way: if $G + G \leq G$, then we have two partitions of $G$, as $\bigcup_{i}X_{i}$ and $\bigcup_{j}Y_{j}$ as well as  $g_{i}, h_{j}\in G$, such that the sets  $g_{i}^{-1}X_{i}$, $h_{j}^{-1}Y_{j}$ are \emph{all} pairwise  disjoint.
\end{proof}

Hence the question of whether a non definably amenable group $G$ has a definable $cpd$ is the same as asking whether $2G \leq G$. (Of course when $G$ is equipped with predicates for all subsets then this has a positive answer, by Tarski's theorem.)
We expect it has a negative answer in general. 

\begin{Remark} Let $G$ be the definable group produced in Section 3 above. Then $4G \leq 3G$. 
\end{Remark}
\begin{proof}  We have (with notation as in Section 3), that $G = \bigcup _{j\in [3]} a(i,j)^{-1}C_{i}$ for each $i \in [4]$.
By cutting down each $C_{i}$ we may assume that for each $i$, the $a(i,j)^{-1}C_{i}$ are disjoint. (Of course the $C_{i}$'s remain disjoint although their union may no longer equal $G$.) 

Now we obtain an injective piecewise definable translation from $4G$ to $3G$ by taking $a(i,j)^{-1}C_{i}$ in the $ith$ copy of $G$ to  $C_{i}$ in the $jth$ copy of $G$, for $i\in [4]$, $j\in [3]$
\end{proof}

It is likely that the generic nature of the example from Section 3 implies that $n= 3$ is least such that $(n+1)G\leq nG$.

In the rest of this subsection we will explain how to modify the example so as to produce $2G\leq G$ also in an ambient $SU$-rank $1$ theory.  So this will be in a sense, a ``better" example,  with respect to the invariant ``least possible $n$" where $n$ is as in Corollary \ref{cor dpd iff n+1<n}. Thus, in this modified example there is a definable $cpd$ and we will see below that the ``definable Tarski number"  (the least sum $m+n$ that can appear in a definable $cpd$ of $G$) equals $6$ which is the least possible for non definably amenable groups definable in simple theories. 

We do a similar thing to Section \ref{sec 3}, but with $F_{6}$ in place of $F_{12}$ and six colours in place of four colours. We choose $a_{i}$ for $i=1,\ldots,6$ to be elements of $SL_{2}(\Z)$ which are free generators of a copy of $F_{6}$ inside $SL_{2}(\Z)$. 
For the universal theory $T$ in Section \ref{subs 3.1} we work in the language of rings with $6$ additional $4$-ary predicates 
$C_{1},\ldots,C_{6}$ and the axioms say  that  $R$ is an integral domain of characteristic $0$, that $C_{1},\ldots,C_{6}$ partition 
$SL_{2}(R)$, and $SL_{2}(R) = a_{1}^{-1}C_{1}\cup a_{2}^{-1}C_{2}\cup a_{3}^{-1}C_{3} = a_{4}^{-1}C_{4}\cup 
a_{5}^{-1}C_{5} \cup a_{6}^{-1}C_{6}$. 

Note that this will already give a definable $cpd$ of $G = SL_{2}(R)$ and so $2G\leq G$ by Lemma \ref{lem def cpd iff 2<1}.

We have to construct again the model companion $T^{*}$ of $T$ and show it to be simple (of $SU$-rank $1$). 

The main thing is to modify the combinatorial lemmas in Section \ref{subs 3.2}. 
 So now we have  a free action of $F_{6}$ on a set $X$, and for $X_{0}\subseteq X$, by a good colouring $c:X_{0} \to [6]$ we mean that for all $x\in X_{0}$ 
\begin{enumerate}[(i)]
\item if $a_{i}\cdot x\in X_{0}$ for all $i=1,2,3$ then $c(a_{i}\cdot x) = i$ for some $i=1,2,3$, and 
\item if $a_{j}\cdot x \in X_{0}$ for all  $j=4,5,6$ then $c(a_{j}\cdot x) = j$ for some $j=4,5,6$.
\end{enumerate}

Again we have the notions of distance, connectedness etc., with respect to the relevant Cayley graph on $X$.

\begin{uLemma}{3.2'}\label{lemma 3.2prim} 
Suppose $X_{0}\subseteq X$ is connected. Then any good colouring $c_{0}: X_{0}\to [6]$ extends to a total good colouring.
\end{uLemma} 
\begin{proof}
(By induction on $n$.) Suppose we have extended the good colouring $c$ of $X_{0}$ to a good colouring $c_{n}$ of $B_{n}(X_{0})$.
Suppose $i\in [6]$, and $y = a_{i}\cdot x$ is in $B_{n+1}(X_{0})\setminus B_{n}(X_{0})$ for some $x\in B_{n}(X_{0})$, then define $c_{n+1}(y) = i$. This is well-defined by uniqueness of paths.  And if $y\in B_{n+1}(X_{0})\setminus B_{n}(X_{0})$ is not of this form, define $c_{n+1}(y)$ arbitrarily. 

Again we have to check that $c_{n+1}$ is a good colouring of $B_{n+1}(X_{0})$.  Suppose $x\in B_{n+1}(X_{0})$ and  
$a_{i}\cdot x\in B_{n+1}(X_{0})$ for all $i=1,2,3$.  If $a_{i}\cdot x \in B_{n}(X_{0})$ for all $i=1,2,3$, then  
connectedness of $B_{n}(X_{0})$ implies that also $x\in B_{n}(X_{0})$. So as $c_{n}$ is a good colouring of $B_{n}
(X_{0})$, and $c_{n+1}$ extends $c_{n}$, Axiom (i) is satisfied at $x$. Otherwise $a_{i}\cdot x \in B_{n+1}(X_{0})
\setminus B_{n}(X_{0})$ for some $i=1,2,3$ and so $x\in B_{n}(X_{0})$, hence  $c_{n+1}(a_{i}\cdot x) = i$.  

Exactly the same holds for $x\in B_{n+1}(X_{0})$ for which $a_{i}\cdot x\in B_{n+1}(X_{0})$ for all $i = 4,5,6$.

So, as in Lemma \ref{lem extending color 1 in sec 3}, we have extended the good colouring of $X_{0}$ to a good colouring of $X$. 
\end{proof}

\begin{uLemma}{3.3'}\label{lemma 3.3prim}
Suppose $C_{0}$ and $C_{1}$ are disjoint connected subsets of $X$ with $3 \leq d(C_{0}, C_{1}) < \infty$. Let $C$ be the smallest connected subset of $X$ containing $C_{0}\cup C_{1}$. Then any good colouring $c_{0}: C_{0}\cup C_{1} \to [6]$ extends to a good colouring of $C$.
\end{uLemma}
\begin{proof}
We have $C_{0}$, $C_{1}$ connected subsets of $X$ with $3\leq d(C_{0}, C_{1}) < \infty$ and $C$ is the smallest connected set containing $C_{0}\cup C_{1}$.  And we want to extend a good colouring $c_{0}$ of $C_{0}\cup C_{1}$ to a good colouring $c$ of $C$.  As in Lemma \ref{lem extending color 2 in sec 3} we reduce to the case of  a path $(u,v,y,z)$ from $u\in C_{0}$ to $z\in C_{1}$ where $v,y\notin C_{0}\cup C_{1}$.  If $v = a_{i}\cdot u$ for some $i\in [6]$, put $c(v) = i$, and define it arbitrarily otherwise. Likewise if $y = a_{i}\cdot z$ for some $i\in [6]$ put $c(y) = i$, and define it arbitrarily otherwise.  Again we check that $c$ is well-defined and that the good colouring axioms are satisfied. 
\end{proof}

Lemmas \ref{lem extending color 3 in sec 3} and \ref{lem extending color 4 in sec 3} adapt (formally) word for word to the new context.  As well as the definition of the model companion $T^{*}$ in Section \ref{subs 3.3} and the simplicity of (all completions of) $T^{*}$ in Section \ref{subs 3.4}.

So the conclusion is:
\begin{Proposition} There is a definable group $G$ in a model of a simple theory such that non definable amenability of $G$ is witnessed by a definable $cpd$, equivalently such that $2G\leq G$. 
\end{Proposition} 

A final remark in this section concerns the numbers $m,n$ witnessing a definable $cpd$, namely the existence of pairwise disjoint definable subsets $X_{1},\ldots,X_{m}$, $Y_{1},\ldots,Y_{n}$ of $G$ and $g_{1},\ldots,g_{m}, h_{1},\ldots,h_{n}\in G$ such that 
$G = \cup g_{i}X_{i} = \cup h_{j}Y_{j}$. Following classical terminology, for a definable group $G$ which is not definably amenable, a least possible value of $m+n$  that occurs in a definable $cpd$ of $G$ can be called  the {\em definable Tarski number} of $G$. (And if $G$ has {\em no} definable $cpd$ we will say that its definable Tarski number is $\infty$.)

\begin{Proposition}  Suppose $G$ is a definable group in a structure $M$ and $G$ has a definable $cpd$ with attached numbers $m, n$. If either $m=2$ or $n=2$ then $Th(M)$ has the strict order property. In particular, the definable Tarski number of a non definably amenable group definable in a simple theory is at least $6$. 
\end{Proposition}
\begin{proof} So we assume that $G$ is the disjoint union of nonempty $X_{1}, X_{2}$, and $Y$ and that $G = g_{1}X_{1}\cup g_{2}X_{2}$.  Then $G$ is also the disjoint union of $g_{1}X_{1}$, $g_{1}X_{2}$ and $g_{1}Y$.
Replacing $X_{1}, X_{2}, Y$ by their $g_{1}$-translates, and changing notation there is $g\in G$ such that $X_{1}\cup gX_{2} = G$.  So $X_{2}$ is a proper subset of $gX_{2}$. Iterating we have a strictly increasing sequence $X_{2}\subset gX_{2}\subset g^{2}X_{2} \subset g^{3}X_{2} \subset \ldots$, yielding the strict order property. 
\end{proof} 

Thus in the modified example above, the definable Tarski number is $6$ (as there is a definable $cpd$ with $n = m = 3$). So in terms of definable  Tarski number, this is the simplest possible example of a non definably amenable group definable in a simple theory. 

\subsection{Small theories}
The aim here is to give some positive results concerning definable amenability for groups definable in (models of) small theories $T$, as well as some related results around amenability of small theories. 

Recall that a complete countable theory $T$ is said to be {\em small} if for all $n \in \mathbb{N}_{\geq 1}$ the type space $S_{n}(T)$ is countable. This is equivalent to saying that for any model $M$ of $T$ and finite subset $A$ of $M$, the type space $S_{1}(A)$ is countable.

We could prove the definable amenability of definable groups in small theories directly from Fact \ref{fact def amen iff no dpd}. But we can slightly generalize the set-up so as to obtain other corollaries. 

Our general context consists of a group $G$ acting on a set $S$ and where we are given a Boolean algebra $\mathcal B$ of subsets of $S$ which is closed under the action of $G$ (in particular $\emptyset$ and $S$ are elements of ${\mathcal B}$).  We will call ${\mathcal B}$ a {\em $G$-invariant Boolean algebra} of subsets of $S$. 

Replacing ``definable" by ``in $\mathcal B$", we can copy the notions of ($m$-)cycles and  ${\mathcal B}$-piecewise translations from Section \ref{subs 4.1} to the present context. 
We can also introduce the notion of \emph{${\mathcal B}$-map} $f$ from a cycle $\sum_{i} X_{i}$ to a cycle  $\sum_{j}Y_{j}$. This will be a map from the formal disjoint union $\bigsqcup_{i}X_{i}$ to the formal disjoint union $\bigsqcup_{j}Y_{j}$ such that for every $i$ and $B\in {\mathcal B}$ with $B\subseteq X_{i}$, and for every $j$, $f(B)\cap Y_{j}\in {\mathcal B}$. Note that this makes sense without any $G$-action.  Note also that any ${\mathcal B}$-piecewise translation is a ${\mathcal B}$-map, and that injectivity makes sense for ${\mathcal B}$-maps. 

Observe that both the class of $\mathcal B$-piecewise translations and the class of $\mathcal B$-maps are closed under composition. As before we write $X\leq Y$ if there is an injective $\mathcal B$-piecewise translation from $X$ to $Y$. 

By a {\em ${\mathcal B}$-paradoxical decomposition}  (${\mathcal B}pd$) we mean an injective ${\mathcal B}$-piecewise translation from $S + Y$ to $Y$ for some cycle $Y$.
  Also we say that the $G$-set $S$ is {\em ${\mathcal B}$-amenable}  if there is a $G$-invariant finitely additive probability measure on ${\mathcal B}$.  The proof of Fact \ref{fact def amen iff no dpd} in \cite{NIPI} adapts to yield:
\begin{Proposition}\label{prop Bamen iff no Bpar dec} The $G$-set $S$ is ${\mathcal B}$-amenable if and only if $S$ has no ${\mathcal B}$-paradoxical decomposition. 
\end{Proposition}

Lemma \ref{lem par dec basics} and Corollary \ref{cor dpd iff n+1<n} remain valid in the more general context of $\mathcal B$-piecewise translations. We use this generalization of Lemma \ref{lem par dec basics}(ii) in the proof of Proposition \ref{prop small boolean algebra} below. 

We will call a Boolean algebra ${\mathcal B}$ {\em small} if  its Stone space is countable, in other words there are only countably many ultrafilters on $\mathcal B$. 

Let us introduce some notation for cycles which will be used in a proof below.  The context here and in the next lemma is simply a Boolean algebra ${\mathcal B}$ of subsets of a set $S$.
Let $X = \sum_{i=1,\ldots,n}X_{i}$ and $Z = \sum_{i=1,\ldots,n}Z_{i}$ be cycles   (so $n$ is the same in both).
We also fix the ordering of the $X_{i}$ and $Z_{i}$.

\begin{enumerate}[(1)]
\item  $X\sqsubseteq Z$ means that $X_{i}\subseteq Z_{i}$ for each $i$,
\item  $X\cap Z = \emptyset$ means that $X_{i}\cap Z_{i} = \emptyset$ for each $i$, and
\item  $X\neq \emptyset$ means that some $X_{i}$ is nonempty.
\end{enumerate}

\noindent
If moreover $f$ is a ${\mathcal B}$-map from $X$ to $Z$, then by the {\em image}  $f(X)$ of $X$ under $f$ we mean the cycle  $\sum_{i}W_{i}$ where $W_{i}$ is  $f(X_{1}\sqcup \ldots\sqcup X_{n})\cap Z_{i}$ which we note is in $\mathcal B$.

\begin{Lemma}\label{lem tree of sets}  Suppose that $Y$ is a nonempty cycle and $f_{0}$, $f_{1}$ are injective $\mathcal B$-maps from $Y$ to $Y$ such that $f_{0}(Y) \cap f_{1}(Y) = \emptyset$. Then ${\mathcal B}$ is not small. 
\end{Lemma}
\begin{proof} The proof goes by induction on the length of the cycle $Y$.  First suppose that $Y$ is a $1$-cycle, i.e. $Y$ is in ${\mathcal B}$.
For $\eta \in 2^{<\omega}$, let $f_{\eta}: Y \to Y$ be  given  by:  $f_{\emptyset}$ is the identity, and $f_{\eta} = f_{\eta(0)}\circ f_{\eta(1)} \circ \ldots \circ f_{\eta(k-1)}$  when 
$dom(\eta) = \{0,\ldots,k-1\}$ with $k > 0$.  And let $Y^{\eta} = f_{\eta}(Y)$.
Then the $Y^{\eta}$ are nonempty subsets of $Y$ which are in ${\mathcal B}$ (as the $f_{i}$ are ${\mathcal B}$ maps), $Y^{\eta} \supset Y^{\tau}$ when $\tau$ extends $\eta$, and  $Y^{\eta 0} \cap Y^{\eta 1} = \emptyset$ for all $\eta$. 
For $\eta\in 2^{\omega}$  let $\Sigma_{\eta} = \{Y^{\eta|n}:n < \omega\}$. Then each $\Sigma_{\eta}$ extends to an ultrafilter $p_{\eta}$ on ${\mathcal B}$ and $p_{\eta} \neq p_{\tau}$ for $\eta \neq \tau \in 2^{\omega}$.  

When $Y$ is an $n$-cycle $\sum_{i=1,\ldots,n}Y_{i}$ for $n >1$  it is a bit more complicated.  With the notation introduced above,  define the $f_{\eta}:Y\to Y$ in  the same way for $\eta\in 2^{< \omega}$, and define  $Y^{\eta} = f_{\eta}(Y)$,  and now define $Y^{\eta}_{i}$ to be $f_{\eta}(Y_{1}\sqcup \ldots \sqcup Y_{n}) \cap Y_{i}$. 

Again the sets $Y^{\eta}_{i}$ are in ${\mathcal B}$ and we have, for all $\eta\in 2^{<\omega}$:
\begin{enumerate}[(1)]
\item  $Y^{\eta} = \sum_{i=1,\ldots,n}Y^{\eta}_{i}$,
\item $Y^{\eta} \neq \emptyset$,
\item  $Y^{\eta}\sqsubseteq Y^{\tau}$ whenever $\eta$ extends $\tau$, and
\item  $Y^{\eta 0} \cap Y^{\eta 1} = \emptyset$. 
\end{enumerate}

 Note that in particular the $Y^{\eta}_{n}$ satisfy both (3) and (4). If they also satisfy  (2) then we get continuum many ultrafilters on ${\mathcal B}$ as in the $n=1$ case. Otherwise there is $\eta$ 
such that $Y^{\eta}_{n} = \emptyset$ and therefore so is $Y^{\eta'}_{n}$ for all $\eta'$ extending $\eta$.  Consider the tree of cycles  $(Y'_{\tau})_{\tau}$ where
$Y'_{\tau} = \sum_{i=1,\ldots,n-1} Y_{i}^{\eta\tau}$.
 These are $(n-1)$-cycles and so by the  inductive hypothesis, we obtain continuum many ultrafilters on ${\mathcal B}$.
\end{proof}

\begin{Proposition}\label{prop small boolean algebra} Let  $S$ be a $G$-set with a $G$-invariant Boolean algebra $\mathcal B$ of subsets.  Suppose that for every finitely generated subgroup $G_{0}$ of $G$ and finite subset ${\mathcal  B}_{0}$ of $\mathcal B$ the Boolean algebra $\langle G_{0}{\mathcal B}_{0} \rangle$ generated by all translates of elements of $B_{0}$ by elements of $G_{0}$ is small. Then $S$ is ${\mathcal B}$-amenable. 
In particular if $\mathcal B$ is small, $S$ is ${\mathcal B}$-amenable. 
\end{Proposition}
\begin{proof} If $S$ is not ${\mathcal B}$-amenable then it is witnessed by  $S+Y \leq Y$ for some nonempty cycle $Y$. By the generalization of Lemma \ref{lem par dec basics}(ii) mentioned after the statement of Proposition \ref{prop Bamen iff no Bpar dec}, we get $2Y \leq Y$, so we have injective ${\mathcal B}$-piecewise translations 
$f_{0}:Y\to Y$ and $f_{1}:Y\to Y$ such that $f_{0}(Y) \cap f_{1}(Y) = \emptyset$.  Let $G_{0}$ be the subgroup of $G$ generated by the finitely many elements of $G$ appearing in the translations in $f_{0}, f_{1}$, and let ${\mathcal B}_{0}$ be the finite collection of elements of ${\mathcal B}$ which appear as the subsets of the elements of the cycle $Y$ which are translated in the maps $f_{0}$, $f_{1}$.  Then $f_{0}$ and $f_{1}$ are $ \langle G_{0}{\mathcal B}_{0} \rangle$-maps, so $\langle G_{0}{\mathcal B}_{0} \rangle$ is not small  by Lemma \ref{lem tree of sets}. Hence, ${\mathcal B}$ is not small. 
\end{proof}

Here are some applications:

\begin{Corollary}\label{cor small implies def amen}  Suppose that $G$ is a definable group in a model $M$ of a small theory $T$. Then $G$ is definably amenable.
\end{Corollary}
\begin{proof}  First as $T$ remains small after naming finitely many parameters we may assume $G$ is $\emptyset$-definable. Remember that definable amenability of $G$ refers to there being a translation invariant Keisler measure on the family of all definable, with parameters in $M$, subsets of $G$.  We apply Proposition \ref{prop small boolean algebra} to the case $S = G$ and 
${\mathcal B}$ the collection of definable subsets of $G$.  If ${\mathcal B}_{0}$ is a finite subset of ${\mathcal B}$ and $G_{0}$ is a finitely generated subgroup of $G$ then the elements of the Boolean algebra  $\langle G_{0}{\mathcal B}_{0} \rangle$ are all definable over a fixed finite set $A$ of parameters. So by smallness of $T$ this Boolean algebra is small, and we can apply Proposition \ref{prop small boolean algebra}. 
\end{proof} 

Proposition \ref{prop small boolean algebra} also gives another proof of  Fact \ref{fact stable gps are dfn amen}  above:
\begin{Corollary}  Any group $G$ definable in a model $M$ of a stable theory is definably amenable.
\end{Corollary}
\begin{proof}  By Fact \ref{fact def amen iff no dpd} we may assume that $T$ is countable and $M$ is countable.  Then for any finite collection $X_{1},\ldots,X_{n}$ of definable subsets of $G$, the Boolean algebra generated by the set of  all left $G$-translates of the $X_{i}$ is small. (For any finite collection of $L$-formulas $\phi_{1}(x,y_{1}), \ldots, \phi_{n}(x,y_{n})$ where $x$ ranges over $G$ and the $y_{i}$ are arbitrary tuples, the Boolean algebra generated by instances $\phi_{i}(x,a_{i})$ of the $\phi_{i}$, with $a_{i}$ in a given countable model  is small.)  So we can apply Proposition \ref{prop small boolean algebra} again. 
\end{proof}

One could unify the two previous corollaries as follows. Let $G$ be a definable group. Suppose that for every finite set $\Delta = \{\phi_{1}(x,y_{1}), \ldots, \phi_{n}(x,y_{n})\}$ of $L$-formulas, and finite set $A$ of parameters, the Boolean algebra of subsets of $G$ which are both $\Delta$-definable and $A$-definable, is small. Then $G$ is definably amenable. 

It is natural to ask whether every complete countable small theory $T$ is {\em amenable}, as defined in the introduction.  However the theory of the dense circular ordering is $\omega$-categorical, with a unique $1$-type over $\emptyset$, but there is no automorphism invariant Keisler measure on the universe $x=x$, as explained in Remark 2.2 of \cite{HKP}, as $\emptyset$ is not an extension base. 

But we point out that a rather weaker property follows from Proposition \ref{prop small boolean algebra}:

$(*)$   For every $\emptyset$-definable set $D$ there is a global Keisler measure concentrated on $D$ which is invariant under {\em definable} automorphisms.
 
\smallskip
We may want to call a complete theory $T$ {\em weakly amenable} if it satisfies $(*)$, but this would be an unnecessary introduction of new terminology. In any case:

\begin{Corollary}\label{cor small theory satisfies (*)}
Suppose that the countable complete theory $T$ is small. Then $T$ satisfies $(*)$.
\end{Corollary}
\begin{proof} Let $D$ be a $\emptyset$-definable set in a saturated model $\bar M$ of $T$. Let $Aut_{def}({\bar M})$ be the group of automorphisms of ${\bar M}$ which are
definable (with parameters) in ${\bar M}$.  Apply Proposition \ref{prop small boolean algebra} to the situation where $G = Aut_{def}({\bar M})$,  $S = D$, and ${\mathcal B}$ is the Boolean algebra of all definable (with parameters) subsets of $D$. Then by smallness the assumption of Proposition \ref{prop small boolean algebra} is satisfied, so we get $(*)$.
\end{proof}

\begin{Remark} 
\begin{enumerate}[(i)]
\item Corollary \ref{cor small theory satisfies (*)} implies Corollary \ref{cor small implies def amen}  via the usual trick of adding a new affine sort.
\item  We obtain a characterization of when an $\emptyset$-definable set $D$ satisfies $(*)$, by the nonexistence of an appropriate  paradoxical decomposition. This is by Proposition \ref{prop Bamen iff no Bpar dec} applied to $G = Aut_{def}({\bar M})$, $S = D$ and ${\mathcal B}$ the Boolean algebra of definable subsets of $D$.
\item Similarly taking $G = Aut({\bar M})$, $S = D$ and ${\mathcal B}$ as in (ii) we obtain a characterization of when there is an automorphism invariant global  Keisler measure on $D$. 
\end{enumerate}
\end{Remark}

Finally we give an application of the discussion above (Lemma \ref{lem tree of sets}) to prove the nontriviality of {\em graded} Grothendieck rings of small theories.  We first recall the usual Grothendieck rings (\cite{K-S}) attached to a structure $M$ which may be many sorted, although we give a slightly different presentation.  We will assume that some sort has at least $2$ elements.
Let $Def(M)$ be the collection of all definable (with parameters from $M$) subsets of products of the basic sorts of $M$.   Let 
$K(M)$ be the free abelian  monoid generated by $Def(M)$. We  can view the elements of $K(M)$ as cycles  $
\sum_{i=1,\ldots,n} X_{i}$ where the $X_{i}$ are definable sets.  As earlier we have the notion of a definable map between 
cycles and in particular  a definable bijection between cycles. Let $\sim$ be  the equivalence relation on $K(M)$ of being in 
definable bijection, for a cycle $D$ let $[D]$ be its $\sim$-equivalence class, and let $K_{semi}(M)$ be the quotient 
$K(M)/\sim$.  In this context one sees that every cycle is $\sim$-equivalent to a definable set (in an appropriate product of sorts), whereby $K_{semi}(M) = \{[D]: D
\in Def(M)\}$, and is moreover an abelian monoid with $0 = [\emptyset]$.  It also has a unital commutative semiring 
 structure by defining $[D_{1}]\cdot [D_{2}] = [D_{1}\times D_{2}]$ and taking the multiplicative identity to be 
$[\{a\}]$ for any singleton in any sort.  Finally we put an  equivalence relation $\sim_{0}$ on $K_{semi}(M)$:  
$[D_{1}] \sim_{0} [D_{2}]$ if there is $[D]$ such that $[D_{1}] + [D] = [D_{2}] + [D]$.   We denote the quotient by 
$K'_{semi}(M)$, a cancellative, unital, commutative semiring. We let $[D]_{0}$ denote the $\sim_{0}$-class of $[D]$.  
Then adding formal additive inverses yields a canonical unital commutative ring $K_{0}(M)$ extending
$K'_{semi}(M)$, called the \emph{Grothendieck ring} of the structure $M$. The elements of $K_{0}(M)$ can be written in the form $[D_{1}]_{0} - [D_{2}]_{0}$, for $D_{1}, D_{2}\in Def(M)$. 
\begin{Example}  Let  $M = (\N, s)$, where $s$ is the successor function. Then $Th(M)$ is small, but $K_{0}(M)$ is trivial.
\end{Example} 
\begin{proof}  The function $s$ gives a definable bijection from $\N$ to $\N\setminus \{0\}$ whereby $[\{0\}]_{0} + [\N]_{0} = [\N]_{0}$ in $K_{0}(M)$, hence $[\{0\}]_{0}$ is the zero element of $K_{0}(M)$. As it 
is also the $1$ of $K_{0}(M)$, $K_{0}(M)$ is trivial. 
\end{proof} 

 However by working with \emph{graded} Grothendieck rings we get a rather different situation.  Again we fix a structure $M$, maybe many-sorted, but we define $K_{0}(S)$ for  $S$ a sort, or product of sorts, and define $K_{0}^{grad}(M)$ to be $\bigoplus_{S}K_{0}(S)$.   Here are the details.   First fix a sort $S$ (or product of sorts).   Start with $Def(S)$ the Boolean algebra of definable (with parameters) subsets of $S$.  Again define $K(S)$ to be collection of cycles of formal sums of elements of $Def(S)$, and $\sim$ the equivalence relation on $K(S)$ of being in definable bijection. Let $K_{semi}(S)$ be the quotient $K(S)/\sim$. It is no longer true that every element of $K_{semi}(S)$ is of the form $[D]$ for $D\in Def(S)$.   Again form $K'_{semi}(S)$, and $K_{0}(S)$ whose elements are of the form $[D_{1}]_{0} -  [D_{2}]_{0}$ for $D_{1}, D_{2}\in K_{semi}(S)$.  Now  $K_{0}(S)$ is just a commutative group with no ring structure.
We define  $K_{0}^{grad}(M)$ to be the direct sum  $\bigoplus_{S}K_{0}(S)$ with its abelian group structure, but also with a commutative ring structure obtained as follows:  for $D_{1}\in Def(S_{1})$ and $D_{2}\in Def(S_{2})$, let $[D_{1}]_{0}\cdot [D_{2}]_{0} =  [D_{1}\times D_{2}]_{0}$ in $K_{0}(S_{1}\times S_{2})$.  And extend it bilinearly to $\cdot$ from $K_{0}(S_{1})\times K_{0}(S_{2})$  to $ K_{0}(S_{1}\times S_{2})$. 

\begin{Proposition} Let $M$ be any structure in a countable language such that $Th(M)$ is small. Then
\begin{enumerate}[(i)]
\item  For every sort $S$, the group $(\Z,+)$ embeds into the group $K_{0}(S)$, in particular $K_{0}(S)$ is nontrivial,
\item The ideal $t\Z[t]$ in the polynomial ring $\Z[t]$ embeds in $K_{0}^{grad}(M)$.
\end{enumerate}
\end{Proposition}
\begin{proof} (i) Consider the homomorphism from $(\Z,+)$ to $K_{0}(S)$ which takes $1$ to $[S]_{0}$. To show it is an embedding we have only to show that for each 
$n\geq 1$, $n[S]_{0} \neq 0$.  Otherwise there is a cycle $Y \in K(S)$ such that  $n[S] + [Y] = [Y]$, yielding a definable injection from $S+Y$ to $Y$, and thus from $Y+Y$ to $Y$ by the obvious variant of Lemma \ref{lem par dec basics}(ii). So we have definable injections $f_{0}:Y\to Y$ and $f_{1}:Y\to Y$ with $f_{0}(Y) \cap f_{1}(Y) = \emptyset$. Let $\mathcal B$ be the Boolean algebra of subsets of $S$ which are definable over the fixed finite set of parameters over which $f_{0}$, $f_{1}$ and the summands of $Y$ are defined. Then $f_{0}$, $f_{1}$ are ${\mathcal B}$-maps. By Lemma \ref{lem tree of sets}, ${\mathcal B}$ is not small,  hence $Th(M)$ is not small, a contradiction.

(ii) Given a sort $S$, we see from the proof of (i) that $[S]_{0}$ generates a subring isomorphic to $t\Z[t]$.  
\end{proof}

Note that we obtain a characterization of when $K_{0}^{grad}(M)$ is trivial. It is precisely that for every sort $S$ and definable subset $D$ of $S$ there is a cycle $Y\in K(S)$ such that 
$D+Y \sim Y$.

On the other hand, triviality of the Grothendieck ring $K_{0}(M)$ is equivalent to there being a definable set $D$ and $d\in D$ such that $D$ and $D\setminus\{d\}$ are in definable bijection. 

A final remark is that the definition of $K_{0}^{grad}(M)$ depends on the choice of sorts $S$.  We could rechoose all definable sets to be sorts, in which case the new graded Grothendieck ring will be bigger and nontrivial, because for a singleton sort $S$, all cycles on $S$ are finite sums of singletons and two such cycles are $\sim$-equivalent iff they have the same cardinality.

\end{document}